\documentclass{aic}

\aicAUTHORdetails{%
  title = {Transversal factors and spanning trees}, 
  author = {Richard Montgomery, Alp M\"uyesser and Yani Pehova},
  plaintextauthor = {Richard Montgomery, Alp Muyesser, Yani Pehova},
}   

\aicEDITORdetails{%
year={2022},
number={3},
received={24 September 2021},
published={16 May 2022},
doi={10.19086/aic.2022.3},
}

\usepackage{enumitem}
\usepackage{totcount}
\regtotcounter{page}

\newcommand\eps{\varepsilon}
\newcommand\N{\mathbb{N}}
\newcommand\E{\mathbb{E}}
\renewcommand\P{\mathbb{P}}

\newcommand\fG{\textbf{\textup{G}}}

\theoremstyle{plain}
\newtheorem{theorem}{Theorem}[section]
\newtheorem{lemma}[theorem]{Lemma}
\newtheorem{proposition}[theorem]{Proposition}
\newtheorem{claim}[theorem]{Claim}
\newtheorem{corollary}[theorem]{Corollary}
\theoremstyle{definition}
\newtheorem{definition}[theorem]{Definition}
\newtheorem{fact}[theorem]{Fact}

\begin{document}

\begin{frontmatter}[classification=text]


\author[richard]{Richard Montgomery\thanks{Supported by the European Research Council (ERC) under the European Union Horizon 2020 research and
innovation programme (grant agreement No.\ 947978).}}
\author[alp]{Alp M\"uyesser}
\author[yani]{Yani Pehova\thanks{Supported by the European Research Council (ERC) under the European Union Horizon 2020 research and
innovation programme (grant agreement No.\ 648509).}}

\begin{abstract}
Given a collection of graphs $\fG=(G_1, \ldots, G_m)$ with the same vertex set, an $m$-edge graph $H\subset \cup_{i\in [m]}G_i$ is a transversal if there is a bijection $\phi:E(H)\to [m]$ such that $e\in E(G_{\phi(e)})$ for each $e\in E(H)$.
We give asymptotically-tight minimum degree conditions for a graph collection on an $n$-vertex set to have a transversal which is a copy of a graph $H$, when $H$ is an $n$-vertex graph which is an $F$-factor or a tree with maximum degree $o(n/\log n)$.
\end{abstract}
\end{frontmatter}

\section{Introduction}\label{sec:intro}
Many important problems in Extremal Graph Theory can be expressed as subgraph containment problems, asking what conditions on a graph $G$ guarantee that it contains a copy of another graph $H$. For example, Mantel's theorem from 1907 states that if $n\geq 3$, then any $n$-vertex graph with more than $n^2/4$ edges contains a triangle.
When $G$ and $H$ have the same number of vertices, it is natural to impose a minimum degree condition on $G$, as seen in Dirac's theorem from 1952. That is, every graph with $n\geq 3$ vertices and minimum degree at least $n/2$ contains a Hamilton cycle. Similarly, when $r|n$ and $H$ is formed of $n/r$ vertex-disjoint copies of $K_r$ (i.e., when $H$ is a $K_r$-factor),
Hajnal and Szemer\'{e}di \cite{hajnalszemeredi} showed that if $G$ has $n$ vertices and $\delta(G)\geq (1-1/r)n$, then $G$ contains a copy of $H$.

For these results on Hamilton cycles and $K_r$-factors, the minimum degree condition used is exactly best possible. When considering more general classes of graphs $H$, exact minimum degree conditions appear more difficult to obtain, and instead asymptotically best possible results have often been shown. For example, Koml\'os, S\'ark\"{o}zy, and Szemer\'edi \cite{kss} showed that, for each $\alpha>0$, there is some $c>0$ and $n_0\in \N$, such that any graph with $n\geq n_0$ vertices and minimum degree at least $(1/2+\alpha)n$ contains a copy of every $n$-vertex tree with maximum degree at most $cn/\log n$. This is tight up to the value of $c$ and $n_0$. For any $r$-vertex graph $F$, K\"uhn and Osthus determined the smallest value of $\delta$ such that, for each $\eps>0$, there is some $n_0\in \N $ such that any graph with $n\geq n_0$ vertices and minimum degree at least $(\delta+\eps)n$ contains an $F$-factor if $r|n$~\cite{kuhnosthus}. For more results concerning different graphs $H$, see the survey by K\"uhn and Osthus~\cite{kossurvey}.

In this paper, we will consider a generalisation of the subgraph containment problem to the study of transversals, where, roughly speaking, one element is selected from each of a collection of sets so that the set of selected elements (the transversal) has some desired property. This framework was introduced by Aharoni (see~\cite{rainbowMantel}), using the terminology of rainbow colouring. More precisely, we say that $\fG=(G_1,\ldots,G_m)$ is a \textit{graph collection on vertex set $V$} if, for each $i\in [m]$, the graph $G_i$ has vertex set $V$. We denote the size $m$ of the collection by $|\fG|$. Given $\fG=(G_1,\ldots,G_m)$, we say that an $m$-edge graph $H$ on $V$ is a \emph{$\fG$-transversal} if there is an injection $\phi\colon E(H)\to [m]$ such that $e\in G_{\phi(e)}$ for each $e\in E(H)$.

Aharoni, DeVos, de la Maza, Montejano and \v{S}\'{a}mal~\cite{rainbowMantel} showed that if $(G_1,G_2,G_3)$ is a graph collection on $[n]$ with $e(G_i)>(\frac{26-2\sqrt{7}}{81})n^2$ for each $i\in [3]$, then $\fG$ contains a transversal which is a triangle. Interestingly, as shown in~\cite{rainbowMantel}, the constant $\frac{26-2\sqrt{7}}{81}>1/4$ cannot be improved here, and therefore a stronger minimum size condition for the transversal triangle problem is required compared to Mantel's theorem. As mentioned in~\cite{rainbowMantel}, it remains an open problem to prove a comparable result for $K_r$, for values of $r$ above 3 (cf.\ Tur\'an's generalisation of Mantel's theorem).

Confirming a conjecture of Aharoni (see~\cite{rainbowMantel}) and improving on an asymptotically-tight result of Cheng, Wang, and Zhao~\cite{CWZ}, Joos and Kim \cite{jooskim} showed that, if $n\geq 3$, then any $n$-vertex graph collection $\fG=(G_1,\ldots,G_n)$ with $\delta(\fG)\geq n/2$ has a transversal which is an $n$-vertex cycle. By taking each graph in $\fG$ to be the same, we can recover Dirac's theorem, and hence this result is similarly tight and gives a transversal generalisation of Dirac's theorem. Joos and Kim \cite{jooskim} also showed the analogous result for a $\fG$-transversal which is a matching.

In this paper, we determine asymptotically tight minimum degree conditions on a graph collection $\fG$ which guarantee a $\fG$-transversal isomorphic to an $F$-factor, or any specific spanning tree without a very high maximum degree. That is, we give asymptotically tight transversal versions of the theorems of Hajnal and Szemer\'{e}di~\cite{hajnalszemeredi} and K\"uhn and Osthus~\cite{kuhnosthus} on factors, and a transversal generalisation of the theorem of Koml\'os, S\'ark\"{o}zy and Szemer\'edi on spanning trees \cite{kss}. Throughout, we aim to use general methods adaptable to finding transversals that are isomorphic to different graphs, using the results of the classical graph problem as a `black box' as much as possible. After giving our results below on  transversal spanning trees, and transversal factors, we finish this section with a discussion of transversals satisfying additional constraints.

\paragraph{Spanning trees.} Our first result is a direct transversal analogue of Koml\'os, S\'ark\"ozy and Szemer\'edi's theorem on spanning trees in dense graphs~\cite{kss}. That is, defining $\delta(\fG)$ to be the smallest minimum degree among the graphs in $\fG$, we prove the following.

\begin{theorem}\label{thm:rainbowkss}
For each $\alpha>0$, there exist $c>0$ and $n_0\in \N$ such that the following holds for all $n\geq n_0$.

Suppose $\fG$ is a graph collection on $[n]$ with $|\fG|=n-1$ and $\delta(\fG)\geq (1/2+\alpha)n$. If $T$ is an $n$-vertex tree with maximum degree at most $cn/\log n$, then there is a $\fG$-transversal which is isomorphic to $T$.
\end{theorem}

Setting each graph $G_i$ in Theorem~\ref{thm:rainbowkss} to be the same graph $G$ with minimum degree at least $(1/2+\alpha)n$, recovers Koml\'os, S\'ark\"ozy, and Szemer\'edi's theorem; thus, Theorem~\ref{thm:rainbowkss} is similarly tight up to the value of $c$ and~$n_0$.

\paragraph{Transversal factors.}
For each graph $F$, let $\delta_F$ be the smallest real number $\delta\geq 0$ such that, for each $\eps>0$ there is some $n_0$ such that, for every $n\geq n_0$ with $|F|$ dividing $n$, if an $n$-vertex graph $H$ has minimum degree at least $(\delta+\eps)n$, then $H$ contains an $F$-factor.
As recalled above, for each $r\geq 3$, we have that $\delta_{K_r}=1-1/r$ by the result of Hajnal and Szemer\'{e}di~\cite{hajnalszemeredi}, while in general the value of $\delta_F$ was determined by K\"uhn and Osthus~\cite{kuhnosthus}. In most cases, $\delta_F$ is linked to the critical chromatic number, while we always have $1-1/(\chi(F)-1) < \delta_F\leq 1-1/\chi(F)$ (for more details see~\cite{kuhnosthus}).

We will show that, in most cases, the same minimum degree bound is asymptotically sufficient for the existence of $F$-factor transversals, as follows.
\begin{theorem}\label{thm:fullyrainbowfactors}
Let $\eps>0$ and let $F$ be a graph on $r$ vertices with $t$ edges. If $\delta_F\geq 1/2$ or $F$ has a bridge, then let $\delta^T_F=\delta_F$, and otherwise let $\delta^T_F=1/2$. Then, there is some $n_0$ such that the following holds for all $n\geq n_0$.

Suppose $\fG$ is a graph collection on $[rn]$ with $|\fG|=tn$ and $\delta(\fG)\geq (\delta^T_F+\eps)rn$. Then, $\fG$ contains a transversal which is an $F$-factor.
\end{theorem}
As the classical graph problem can be deduced by again taking $\fG$ to be copies of the same graph, it follows from the minimality of $\delta_F$ that to show Theorem~\ref{thm:fullyrainbowfactors} is asymptotically tight, we need only show that if $F$ has no bridge, then the threshold $\delta^T_F$ for the existence of a transversal $F$-factor is at least $1/2$. To see this, suppose $F$ has no bridge, and take two disjoint sets $A$ and $B$ with size $\lfloor rn/2\rfloor$ and $\lceil rn/2\rceil$ respectively. For each $i\in [tn-1]$, let $G_i$ be the union of a clique on $A$ and a clique on $B$. Let $G_{tn}$ be the complete bipartite graph with bipartition $(A,B)$, and let $\fG=(G_1,\ldots,G_{tn})$. If $\fG$ contains a transversal which is an $F$-factor, then the copy of $F$ with an edge in $G_{tn}$ can have only this edge between the sets $A$ and $B$, and thus $F$ has a bridge, a contradiction. Therefore, $\fG$ contains no $F$-factor transversal. As $\delta(G)=\lfloor rn/2\rfloor-1$, it holds that $\delta^T_F\geq 1/2$.

Our methods to prove Theorem~\ref{thm:fullyrainbowfactors} allow some control over which edges in the $F$-factor appear in which graphs in $\fG$. Namely, using the same framework, we also prove the following natural variation of Theorem~\ref{thm:fullyrainbowfactors}.

\begin{theorem}\label{thm:monocopies} Let $\eps>0$ and let $F$ be an $r$-vertex graph. Then, there is some $n_0$ such that the following holds for all $n\geq n_0$.

Suppose $\fG$ is a graph collection on $[rn]$ with $|\fG|=n$ and  $\delta(\fG)\geq (\delta_F+\eps)rn$. Then, there are vertex-disjoint $F$-copies $F_i\subset G_i$, $i\in [n]$.
\end{theorem}

When $F$ is a clique, the above result has connections to the study of cooperative colourings, a variation on graph colourings introduced in \cite{coopcolour} (see also \cite{othercoopcolour}). Given a graph collection $(G_1,\ldots,G_m)$ on $V$, we say that the collection can be cooperatively coloured if there exists a choice of independent sets $I_i\subseteq G_i$ ($i\in [m]$) such that $\bigcup_{i\in[m]}I_i=V$. By taking complements, the $F=K_r$ case of Theorem~\ref{thm:monocopies} could be viewed as giving a sufficient condition for a graph collection to be \textit{equitably} cooperatively coloured, that is, via independent sets of the same size.

\paragraph{Stronger conditions on transversals.} In the problems we have considered, we have typically sought a copy of a graph $H$ in a graph collection where each edge of the copy of $H$ is provided by a different graph in the collection. We could also specify which edge of $H$ should have its copy provided by which graph. We will sketch how our proofs can be modified to show that this is possible in Theorem~\ref{thm:fullyrainbowfactors} if $F$ is a clique (see Section~\ref{sec:withspecifiedcolouring}), but also that in general the minimum degree condition may need to be raised for this to be possible (see Proposition~\ref{prop:noeasygeneralisation}).

\paragraph{Other transversals and our techniques.} Throughout the paper, we use as general techniques as possible to move from embedding graphs using minimum degree conditions to embedding graphs as transversals using minimum degree conditions. These techniques are described in Section~\ref{sketch}. The methods we use form a good starting framework for finding transversals isomorphic to other graphs. For example, due to Koml\'os, S\'ark\"ozy and Szemer\'edi it is known that, for each $k\geq 1$, there is some $n_0$ such that any graph with $n\geq n_0$ vertices and minimum degree at least $kn/(k+1)$ contains the $k$th power of a Hamilton cycle~\cite{ksspower}. It seems likely that the $k$th power of a Hamilton cycle can be found as a transversal in any $n$-vertex graph collection with minimum degree $(k/(k+1)+\alpha)n$, if $n$ is large relative to $k$ and $1/\alpha$, and that this could be shown by adapting our techniques and proving slightly strengthened versions of the non-transversal embedding problem. To see what kind of modifications to the non-transversal embedding problem is required by our approach, see Section~\ref{sec:fixedvertex}.

\medskip

\paragraph{Recent related results.} During the final preparation of this manuscript, Cheng, Han, Wang, and Wang released a manuscript \cite{chenghan} proving Theorem~\ref{thm:fullyrainbowfactors} in the special case that $F$ is a clique, along with analogous results for hypergraph transversals.

\bigskip

In the next section, after covering our notation, we discuss our proofs before giving an outline of the rest of the paper.


\section{Proof sketch and preliminaries}
For convenience, we will use colour terminology for the rest of the paper. Given a graph collection $\fG=(G_1,\ldots,G_m)$, we consider each graph $G_i$ to have a colour $i$. Given a subgraph $H\subset \cup_{i\in [m]}G_i$, we allow each edge $e\in E(H)$ to be assigned the colour $i$ if $e\in E(G_i)$. Then, $H$ is a transversal if each edge can be assigned a different colour. As often in the area of edge colourings, we call graphs with a different colour for each edge \emph{rainbow}. When we say $H\subset \cup_{i\in [m]}G_i$ is \textit{uncoloured}, we are simply highlighting that we have not yet assigned a colouring to $H$.

\subsection{Notation}\label{notation}
We use standard graph theory notation throughout, but describe here some of the more common notation we use. We denote by $[n]$ the set of integers $\{1,...,n\}$, and by $\binom{[n]}{k}$ the collection of subsets of $[n]$ of size $k$.

A graph $G$ has $e(G)$ edges and $|G|$ vertices, while, given any set $V\subset V(G)$ and a vertex $v\in V(G)$, the number of neighbours of $v$ in $V$ is denoted by $d(v,V)$ .
We often describe finding a copy, $S$ say, of a graph $T$ in a graph $G$. When we do so we implicitly assume that we have found an embedding $\phi:T\to G$ with image $S$. Thus, when we say that $s$ is the copy of a vertex $t\in V(T)$ in $S$, we mean that $s=\phi(t)$ without referencing or labelling the function~$\phi$.

Recall that a graph collection $\fG=(G_1,\ldots,G_m)$ is a collection of (not necessarily distinct) graphs $G_i$, $i\in [m]$, which all have the same vertex set, and $\delta(\fG)=\min_{i\in [m]}\delta(G_i)$. Given a graph collection $\fG=(G_1,\ldots,G_m)$ with vertex set $V$, and  a set $U\subset V$, the collection of graphs $G_i[U]$, $i\in [m]$ induced on the vertex set $U$ is denoted by $\fG[U]$. We set $|\fG|$ to denote the size of $\fG=(G_1,\ldots,G_m)$, so that, in this case, $|\fG|=m$.

We use standard hierachical notation for constants, writing $x\ll y$ to mean that there is a fixed positive non-decreasing function on $(0,1]$ such that the subsequent statements hold for $x\leq f(y)$. Where multiple constants appear in a hierarchy, they are chosen from right to left. We omit rounding signs where they are not crucial.

\subsection{Proof sketch}\label{sketch}
We first sketch how we embed spanning trees before discussing the adaptations we make to embed factors. We start by discussing two key techniques, and then give an overview of the structure of the proof.

\paragraph{Rainbow subgraphs using surplus colours.} In the proof of Theorem~\ref{thm:rainbowkss}, we embed the $n$-vertex tree $T$, one small subtree at a time, into randomly sampled vertex subsets of the appropriate size. More precisely, suppose we are trying to embed in a rainbow fashion a subtree $T'\subseteq T$ on a vertex subset $S\subseteq[n]$ where $S$ is sampled randomly amongst $|T'|$-sized subsets. As the graph collection $\fG[S]$ will have $\delta(\fG)\geq (1/2+\alpha/2)|S|$ with high probability, this problem is quite similar to the original problem of embedding $T$ in $\fG$. The key difference is that, assuming $T'$ is small, $|\fG[S]|$ is much larger than we need it to be to find a rainbow embedding of $T'$ on $S$. For example, if $|T'|=o(n)$, the size of $\fG[S]$ is $n-1$, so that $\fG$ has $C$ times more graphs than we need to find a rainbow embedding, where $C$ is a large constant. Hence, the problem of embedding $T'$ is equivalent to proving a version of Theorem~\ref{thm:rainbowkss} where the graph collection has size $Cn$, as opposed to $n-1$.
\par To illustrate how much simpler this problem is in comparison, suppose $1/n\ll 1/C\ll \alpha$, $m= Cn$ and let $\fG=(G_1,\ldots,G_m)$ be a graph collection on $[n]$ with $\delta(\fG)\geq (1/2+\alpha) n$. We let $G$ be the graph with vertex set $[n]$ formed by the edges which appear in at least $2n$ graphs $G_i$. As $m$ is much larger than $2n$, it follows that $G$ has almost as good a minimum degree condition as $\fG$ (see Proposition~\ref{prop:alsodirac}), so that we will have $\delta(G)\geq (1/2+\alpha/2)n$. If we can use an uncoloured embedding result to find a copy of an $n$-vertex tree $T$ in $G$, then we can greedily colour its edges using different colours as each edge has at least $2n$ possible colours in $\fG$.
\par By embedding $T$ as a sequence of subtrees, deleting colours used on each embedded rainbow subtree as we go along, we can iteratively apply the above argument to prove a version of Theorem~\ref{thm:rainbowkss} where the graph collection has size $(1+o(1))n$ as opposed to $Cn$ (see Section~\ref{surpluscolours}).

\paragraph{Colour absorption.} To move from a version of Theorem~\ref{thm:rainbowkss} with $(1+o(1))n$ graphs in $\fG$ to the actual statement of Theorem~\ref{thm:rainbowkss} , we use \textit{absorption}, as first codified by R\"odl, Ruci\'nski and Szemer\'edi~\cite{RRSab}. Suppose first that in $\fG$ with $|\fG|=n-1$ we could embed a small subtree of $T$, say $S$, such that given any $e(S)$-sized subset $\tilde{C}$ of $[n-1]$, we can colour the (uncoloured) image of $S$ using exactly the colours in $\tilde{C}$. Here, we will have $e(S)\ll n-1$. Leaving the image of $S$ uncoloured, we then extend the embedding to $T$ while giving each newly embedded edge a different colour, relying on the $e(S)$ surplus colours we have to do this.  Finally, we take the $e(S)$ unused colours and use them to give a rainbow colouring to $S$. In the language of absorption we say the embedding of $S$ \textit{absorbs} these unused colours.

\par Unfortunately, our embedding of $S$ will not be able to absorb \emph{any} set of $e(S)$ colours. Instead, we find disjoint subsets $C, A\subseteq [n-1]$, such that the embedding of $S$ has a rainbow colouring using exactly the colours in $\tilde{C}\cup A$, for any set $\tilde{C}\subset C$ of $e(S)-|A|$ colours (see \ref{step-1} in the proof overview below). Thus, $C$ functions as a \textit{reservoir} of colours to assist with extending the embedding of $S$ to one of $T$ while colouring the newly embedded edges with different colours in $[n-1]\setminus A$. Here, we will have that $e(S),|A|\ll |C|\ll n-1$.
\par The only additional complication here is that we need to ensure we use every colour outside of $C\cup A$ at some earlier stage in the embedding (see \ref{step-2} and \ref{step-3} below), so that any remaining colour is absorbable by the embedding of $S$.
\par We now discuss how we embed $S$ (and find sets $A,C$) with this absorber-reservoir property. Let $0<1/n\ll \eps\ll \alpha$. Using an auxiliary graph similar to $G$ above, we can embed $S$, an $\eps n$-edge subtree of $T$, into $\fG=(G_1,\ldots,G_{n-1})$, so that each embedded edge is in at least $2\eps n$ graphs in $\fG$. Edge-by-edge, we randomly give each edge of $S$ one of its possible colours, ensuring that this results in a rainbow colouring. Letting $\phi$ be the embedding of $S$, note that if an edge $e$ of $\phi(S)$ has colour $i$, then we can alter the colouring of $\phi(S)$ by recolouring $e$ with any unused colour $j$ with $e\in E(G_j)$. This gives a rainbow colouring of $\phi(S)$ in which the colour $i$ has been switched for $j$, and can be done for at least $\eps n$ values of $j$ for each colour $i$ appearing on the image of $S$.

\par We can now combine many switchings of the above form to build \textit{chains} of switchings with which  we can build the desired reservoir property. Consider an auxiliary digraph $K$ with vertex set $[n-1]$ where $\vec{ij}$ is an edge whenever there is an edge with colour $i$ in $\phi(S)$ which appears in $E(G_j)$, noting that this graph will have $\Omega_\eps(n^2)$ edges. Let $0<\mu\ll \eps$. Suppose we could find a large set $C$ of colours not on $\phi(S)$ and a set $B$ of $\ell:=\mu n$ colours on $\phi(S)$ such that, given any set $\tilde{C}\subset C$ of $\ell$ colours, there are $\ell$ vertex-disjoint directed paths from $B$ to $\tilde{C}$ in $K$. Then, the set $C$ and the set $A$ of colours of $\phi(S)$ not in $B$, together have the property we want. Indeed, given any set $\tilde{C}\subset C$ of $\ell$ colours, we can use $\ell$ vertex-disjoint paths from $B$ to $\tilde{C}$ in $K$ to carry out a sequence of switchings along each path, resulting in $\phi(S)$ being coloured with exactly the colours in $A\cup \tilde{C}$.

This illustrates the mechanism behind our colour absorption technique. In practice, to find the sets $C$ and $B$ with this property, we consider the subgraph $K'\subset K$ of edges $\vec{ij}\in E(K)$ for which we also have $\vec{ji}\in E(K)$. It will be easy to see that the random embedding of $S$ will likely result in $K'$ also having $\Omega_\eps(n^2)$ edges. This allows us to apply a famous theorem of Mader to find a well-connected subgraph of the underlying undirected graph of $K'$, in which we can then use a version of Menger's theorem to find vertex-disjoint paths. Any colours with sufficiently many out-neighbours in $K$ into the well-connected subgraph can then be allocated to the colour reservoir $C$.  More details on this can be found in Section~\ref{sec:colourabsorb}.

\paragraph{Overview.} To embed an $n$-vertex graph $H$ which is a tree or a factor in a rainbow fashion in $\fG$, we use the following overview, which is a common framework for embeddings using absorption. Where the $n$-vertex $m$-edge graph $H$ is a tree or an $F$-factor, we start by dividing $H$ into four subgraphs $H_1\cup H_2\cup H_3\cup H_4$ with carefully chosen sizes, before embedding the subgraphs one by one into $\fG=(G_1,\ldots,G_m)$ in the following five steps.

\begin{enumerate}[label = \textbf{Step \arabic*}]
\item  \label{step-1} \textbf{Find a colour absorber.} We embed $H_1$ into $\cup_{i\in [m]}G_i$, while finding disjoint sets $A,C\subset [m]$ such that given any $e(H_1)-|A|$ colours in $C$, we can give the embedding of $H_1$ a rainbow colouring using exactly those colours and the colours in $A$.
\item  \label{step-2}\textbf{Use most of the colours not in $A\cup C$.} We extend the embedding of $H_1$ to one of $H_1\cup H_2$, colouring the newly embedded edges in a rainbow fashion using most of the colours outside of $A\cup C$.
\item \label{step-3} \textbf{Use the last of the colours not in $A\cup C$.} We extend the embedding of $H_1\cup H_2$ to one of $H_1\cup H_2\cup H_3$, using colours in $C$ as well as every unused colour outside of $A\cup C$ for the newly embedded edges while maintaining a rainbow colouring.
\item \label{step-4} \textbf{Embed the last of the vertices of $H$ using colours in $C$.} We extend the embedding of $H_1\cup H_2\cup H_3$ to one of $H_1\cup H_2\cup H_3\cup H_4$, using unused colours in $C$ for the newly embedded edges while maintaining a rainbow colouring.
\item \label{step-5} \textbf{Use the colour absorber.} Finally, we use the remaining unused colours in $C$ along with those in $A$ to colour the embedding of $H_1$.
\end{enumerate}

For \ref{step-1} we use the colour absorption technique described above (and carried out in Section~\ref{sec:colourabsorb}), and the property we find allows us to later carry out \ref{step-5}. For \ref{step-2} and \ref{step-4} we further divide $H_2$ and $H_4$ into subgraphs so that we can iteratively embed the smaller pieces while having relatively many surplus colours (as described above, and as carried out in Section~\ref{surpluscolours}). For \ref{step-3}, by choosing $H_3$ to be the smallest of $H_1$, $H_2$, $H_3$ and $H_4$, we will ensure that we have many spare vertices and colours (from $C$) at this step to make it as easy as possible to use the last of the colours not in $A\cup C$.

\paragraph{Techniques for factors.} The above techniques for spanning trees also work well for embedding $F$-factors so that each edge is a different colour. However, we need to use a different approach for finding $F$-factors where the edges of each copy of $F$ have the same colour, which is different for each copy of $F$ (that is, for Theorem~\ref{thm:monocopies}).

The main difference is in finding $F$-factors with different coloured copies of $F$ in graph collections where we have more colours than we need. Suppose $1/n\ll 1/C\ll \alpha,1/r$ and $\fG=(G_1,\ldots,G_{Cn})$ is a graph collection on $[rn]$ with $\delta(\fG)\geq (\delta_F+\alpha) rn$, where $F$ is an $r$-vertex graph. Set $m:=Cn$. We wish to find $n$ vertex disjoint copies of $F$ which can be each given a different colour. Taking some integer $k$ with $1/n\ll 1/k\ll 1/C$, consider a random subset $V$ of $[rn]$ with $rk$ vertices. It is very likely that, for most colours $i\in [m]$, it holds that $\delta(G_i[V])\geq (\delta_F+\alpha/2)|V|$ (see Proposition~\ref{lem:hypergeom-app}), and, for each such $i\in [m]$, the graph $G_i[V]$ contains an $F$-factor. As the number of colours ($Cn$) is very large, and the number of different $F$-factors with vertex set $V$ is certainly at most $(rk)!$, we will be able to find an $F$-factor $H$ of $\fG[V]$ such that $H\subset G_i$ for $\Omega_{k,r}(n)$ colours $i\in [m]$. Thus, we can colour each copy of $F$ in $H$ a different colour. Of course, this finds a coloured $F$-factor of $\fG[V]$, not $\fG$. However, by partitioning $[rn]$ into vertex sets $V_1\cup\ldots \cup V_\ell$ with size around $rk$ in which most colours have a good minimum degree condition (see Lemma~\ref{lem:partition}), we can iteratively apply this technique to find such an $F$-factor of $\fG[V_i]$ in turn for $i=1,\ldots,\ell$, each time using new colours (see Section~\ref{sec:facspare}).

Note that picking such a random subset $V$ finds copies of $F$ which can be coloured with many different colours (see Claim~\ref{clm:onecopy}). Therefore, we can use this to develop a colour absorber using the same ideas in the earlier sketch, and in particular apply the same underlying result (Lemma~\ref{lem:absorb}) to an appropriate auxiliary graph.

\paragraph{Paper structure.} In the rest of this section, we give some preliminary results we will use, most notably proving our partitioning lemmas. In Section~\ref{sec:colourabsorb}, we develop our colour absorption structures. In Section~\ref{sec:trees}, we prove Theorem~\ref{thm:rainbowkss}. In Section~\ref{sec:factors}, we prove Theorems~\ref{thm:fullyrainbowfactors} and \ref{thm:monocopies}. Finally, in Section~\ref{sec:withspecifiedcolouring}, we discuss $F$-factors with stronger conditions on the edge colouring.

\subsection{Edges in many graphs in a graph collection}\label{sec:twomain}
We now give the simple proposition that underpins our initial rainbow embedding using uncoloured embedding results, as described in Section~\ref{sketch}.
\begin{proposition}\label{prop:alsodirac} Let $0\leq \alpha, \delta\leq 1$ and $m,n\in \N$.
Let $\fG$ be a graph collection on $[n]$ with $|\fG|=m$ and $\delta(\fG)\geq \delta n$. Let $G$ be the graph with vertex set $[n]$, where $e$ is an edge of $G$ if $e\in E(G_i)$ for at least $\alpha m$ values of $i\in [m]$.
Then, $\delta(G)\geq (\delta-\alpha)n$.
\end{proposition}
\begin{proof}
For each $v\in [n]$, we have
\[
m\cdot \delta n\leq \sum_{i\in [m]}d_{G_i}(v)\leq\
m\cdot d_G(v)+n\cdot \alpha m,
\]
and therefore $d_G(v)\geq (\delta-\alpha)n$. Thus,  $\delta(G)\geq (\delta-\alpha)n$.
\end{proof}

\subsection{Partitioning the vertex set}
Here we will prove our two main results for partitioning the vertex set of a graph collection (Lemmas~\ref{lem:newdivide1} and \ref{lem:partition}). For these we use the following standard concentration result (see, for example, \cite{AS}), followed by a simple proposition on a single vertex subset chosen randomly.

\begin{lemma}\label{hypergeom}
Let $X$ be a hypergeometric random variable with parameters $N$, $n$ and $m$\footnote{A hypergeometric random variable with parameters $N$, $n$ and $m$ takes value $k$ with probability $\binom{m}{k}\binom{N-m}{n-k}/\binom{N}{n}$.}. Then, for any $t>0$,
\[
\P(|X-\E(X)|\geq t)\leq 2e^{-2t^2/n}.
\]
\end{lemma}

\begin{proposition}\label{lem:hypergeom-app}
Let $0<\delta'<\delta$ and $k \le n \in \N$ with $1/k\ll \delta-\delta'$. Let $G$ be an $n$-vertex graph containing a vertex set $V$ with $|V|\geq k$ such that $d_G(v,V)\geq \delta n$ for each $v\in V(G)$, and let $w\in V(G)$. Then, if $A\subseteq V$ is a vertex set of size $k$ chosen uniformly at random, we have
\[
\P \left( \delta(G[\{w\}\cup A])<\delta'k\right)\le e^{-(\delta-\delta')^2k}.
\]
\end{proposition}
\begin{proof} Let $t=(\delta-\delta')k\sqrt{2/3}$.
For each $v\in V$, note that, conditioned on $v\in A$, the random variable $d_{G}(v,A)$ is hypergeometrically distributed with parameters $|V|-1$, $k-1$ and $d_G(v,V)$, and mean $(k-1)\cdot d_G(v,V)/(|V|-1)\geq (k-1)\delta n/(n-1)> \delta'k+t$.
Thus, by Lemma~\ref{hypergeom} applied with $t$ we have
\[
\P\left(d_{G}(v,A)< \delta'k\right|v\in A)\le 2e^{-\frac{4}{3}(\delta-\delta')^2k}.
\]
Similarly, we have that $\P(d_{G}(w,A)<\delta'k) \le 2\exp(-4(\delta-\delta')^2k/3)$.
Therefore, the probability that we have $\delta(G[\{w\}\cup A])<\delta'k$ is at most
\[
\P(d_{G}(w,A)<\delta'k)+\sum_{v\in V} \P(v\in A)\cdot \P\left(d_{G}(v,A)< \delta'k|v\in A\right)\le 2(k+1)e^{-\frac{4}{3}(\delta-\delta')^2k}\le e^{-(\delta-\delta')^2k},\]
as required.
\end{proof}

For embedding spanning trees, we use the following partition result into constantly many vertex sets.

\begin{lemma}\label{lem:newdivide1}
Let $0<1/n\ll \eps,\alpha\leq 1$, and let $\delta>0$, $n\in \N$ and $m\leq n/\eps$. Suppose $\fG=(G_1,\ldots,G_m)$ is a graph collection with vertex set $[n]$ and suppose $V\subset [n]$ such that $d_{G_i}(v,V)\geq (\delta+\eps)n$ for each $i\in [m]$ and $v\in [n]$. Let $k\in \N$, and let $n_1,\ldots,n_k\geq \alpha n$ be integers such that $\sum_{i\in [k]}n_i=|V|$.

Then, there is a partition $V=V_1\cup \ldots \cup V_k$ such that, for each $i\in [k]$, $|V_i|=n_i$ and $d_{G_j}(v,V_i)\geq (\delta+\eps/2)n_i$ for each $j\in [m]$ and $v\in [n]$.
\end{lemma}
\begin{proof} Pick a partition $[n]=V_1\cup \ldots \cup V_k$ uniformly at random from such partitions where, for each $i\in [k]$, $|V_i|=n_i$. As $n_i\geq \alpha n$ for each $i\in [k]$, we have that the probability there is an $i\in [k]$, $j\in [m]$ and $v\in [n]$ with $d_{G_j}(v,V_i)< (\delta+\eps/2)n_i$ is, by Proposition~\ref{lem:hypergeom-app}, at most $k\cdot m\cdot n\cdot \exp(-\eps^2\alpha n/4)=o(1)$. Here we use $1/n\ll \eps, \alpha$. Therefore, there is some partition for which the desired property holds.
\end{proof}

For embedding factors, we use the following partitioning result into many vertex subsets which each have constant size. The proof is by an iterative partitioning, using calculations based on those by Barber, K\"uhn, Lo and Osthus~\cite[Section 7]{BKLO}.

\begin{lemma}\label{lem:partition}
Let $0<1/k\ll \eps\leq 1$, and let $\delta>0$ and $n\in \N$. Suppose $\fG$ is a graph collection on vertex set $[n]$ with $\delta(\fG)\geq (\delta+\eps)n$. Let $\ell\in \N$, and let $n_1,\ldots,n_\ell\geq k$ be integers such that $\sum_{i\in [\ell]}n_i=n$.
Let $m=|\fG|$.

Then, there is a partition $[n]=V_1\cup \ldots \cup V_\ell$ such that, for each $i\in [\ell]$, $|V_i|=n_i$, and, for all but at most $m/k^2$ values of $j\in [m]$, we have $\delta(G_j[V_i])\geq (\delta+\eps/2)n_i$.
\end{lemma}
\begin{proof} For each $r\in \N$ and $\eta\geq 0$, call $\mathcal{I}=(I_1,\ldots,I_s)$ an \emph{$(r,\eta)$-good partition of $[\ell]$} if there is a partition $U_1\cup \ldots \cup U_s$ of $[n]$ satisfying the following.
\begin{enumerate}[label = \textbf{A\arabic{enumi}}]
\item For each $i\in [s]$, $r/2\leq |I_i|\leq r$.\label{prop:0}
\item For each $i\in [s]$, $|U_i|=\sum_{j\in I_i}n_j$.\label{prop:1}
\item For each $i\in [s]$, for all but at most $\eta m/k^2$ values of $j\in [m]$, we have $\delta(G_{j}[U_{i}])\geq \left(\delta+\eps-\eta\right)|U_i|$. \label{prop:2}
\end{enumerate}

Note that $\mathcal{I}=([\ell])$ is an $(\ell,0)$-good partition of $[\ell]$, and, as $\eps\leq 1$, to show the lemma it is sufficient to show that $[\ell]$ has a $(1,\eps/2)$-good partition. We will show such a partition exists by repeatedly applying the following claim.

\begin{claim}\label{clm:gdtogd} For each $m,\bar{m}\geq 1$ with $m/2\leq \bar{m}\leq m\leq \ell$ and $\eta>0$, if $[\ell]$ has an $(m,\eta)$-good partition, then it has an $(\bar{m},\eta+(k\bar{m})^{-1/3})$-good partition.
\end{claim}
\begin{proof}
Let $(I_1,\ldots,I_s)$ be an $(m,\eta)$-good partition of $[\ell]$ with accompanying partition $U_1\cup \ldots \cup U_s$ of $[n]$. For each $i\in [s]$, let $m_i=\lfloor 2|I_i|/\bar{m}\rfloor$, so that, by \ref{prop:0}, $m_i\leq 4$, and let $\mathcal{I}_i$ be a partition of $I_i$ into sets $I_{i,1}\cup \ldots\cup I_{i,m_i}$ with $\bar{m}/2\leq |I_{i,j}|\leq \bar{m}$ for each $j\in [m_i]$.
For each $i\in [s]$, partition $U_i$ as $\cup_{j\in [m_i]}U_{i,j}$ uniformly at random so that $|U_{i,j}|=\sum_{i'\in I_{i,j}}n_{i'}$ for each $j\in [m_i]$. Note that this is possible by \ref{prop:1}.

Let $\bar{\eta}=(k\bar{m})^{-1/3}$. For each $i\in [s]$, let $J_i$ be the set of $j\in [m]$ with $\delta(G_{j}[U_{i}])\geq  \left(\delta+\eps-\eta\right)|U_i|$, so that $|[m]\setminus J_{i}|\leq \eta m/k^2$ by \ref{prop:2}.
For each $j'\in [m_i]$, let $J_{i,j'}$ be the set of $j\in [m]$ with $\delta(G_{j}[U_{i,j'}])\geq \left(\delta+\eps-\eta-\bar{\eta}\right)|U_{i,j'}|$.
Let $p=\exp(-(k\bar{m})^{1/4})$. Then, by Proposition~\ref{lem:hypergeom-app}, for each $j\in J_{i}$, as $|U_{i,j'}|\geq k|I_{i,j'}|\geq k\bar{m}/2$, we have
\begin{align*}
\P(j\notin J_{i,j'})&=\P(\delta(G_{j}[U_{i,j'}])< (\delta+\eps-\eta-\bar{\eta})|U_{i, j'}|)\leq  \exp(-\bar{\eta}^2|U_{i,j'}|)
\leq  \exp(-\bar{\eta}^2k\bar{m}/2)\leq p.
\end{align*}
Therefore, as $|[m]\setminus J_{i}|\leq \eta m/k^2$,  using Markov's inequality we have, for each $j'\in [m_i]$,
\begin{align}
\P(|[m]\setminus J_{i,j'}|\ge (\eta+\bar{\eta})m/k^2)
&\leq
\P(|J_i\setminus J_{i,j'}|\ge \bar{\eta}m/k^2)
\leq \frac{pm}{\bar{\eta}m/k^2}\leq 1/10,\label{eqn:MI}
\end{align}
as $1/k\ll 1$.
For each $i\in [s]$, let $E_i$ be the event that, for each $j\in [m_i]$, we have $|[m]\setminus J_{i,j}|\le (\eta+\bar{\eta})m/k^2$. As $m_i\leq 4$, by \eqref{eqn:MI}, we have that $E_i$ holds with positive probability.

Noting that any events $E_{i}$ and $E_{j}$ are independent if $i\neq j$, we can therefore choose partitions $U_{i,j}$, $i\in [s]$ and $j\in [m_i]$, such that $E_{i}$ holds for each $i\in [s]$. Let $s'=\sum_{i\in [s]}m_i$ and relabel the sets $I_{i,j}$ and $U_{i,j}$, $i\in [s]$ and $j\in [m_i]$, as $I'_1,\ldots,I'_{s'}$ and $U'_1,\ldots,U'_{s'}$ to get an  $(\bar{m},\eta+(k\bar{m})^{-1/3})$-good partition, as required.
\end{proof}

For the appropriate integer $a\leq 2\log_2\ell$, take the shortest sequence of integers $\ell_0=\ell,\ldots,\ell_a=1$ with $\ell_i=\lfloor \ell_{i-1}/2\rfloor$ for each $i\in [a]$. Now, $[\ell]$ has an $(\ell,0)$-good partition, and therefore, by repeated application of Claim~\ref{clm:gdtogd}, it has an $(\ell_a,\sum_{i=0}^a(k\ell_i)^{-1/3})$-good partition. Note that since $\ell_i\ge 2^{a-i}$ for all $i$, we have $\sum_{i=0}^a(k\ell_i)^{-1/3}\leq k^{-1/3}\sum_{i=0}^a2^{-(a-i)/3}=k^{-1/3}\sum_{i=0}^a2^{-i/3}\leq 6k^{-1/3}\leq \eps/2$, as $1/k\ll \eps$. As $\ell_a=1$, we thus have that $[\ell]$ has a $(1,\eps/2)$-good partition, as required.
\end{proof}


\section{Colour absorption}\label{sec:colourabsorb}
We now detail our colour absorption, as described in Section~\ref{sketch}. We start by recalling a famous theorem of Mader and a version of Menger's theorem (see, for example, \cite{moderngraphtheory}), as follows.
\begin{theorem}\label{thm:mader}
Every graph with average degree at least $4k$ has a $(k+1)$-connected subgraph.
\end{theorem}
\begin{theorem}\label{thm:maxflow}
If a graph $G$ is $k$-connected and contains two disjoint subsets $A,B\subset V(G)$ with size $k$, then $G$ contains $k$ vertex-disjoint paths from $A$ to $B$.
\end{theorem}

We now prove a lemma (Lemma~\ref{lem:absorb}), which allows us to find the sets $A$ and $C$ described in Section~\ref{sketch}. We prove this in a more general setting to be applied to an auxiliary bipartite graph. This shows that in an unbalanced dense bipartite graph we can robustly find a matching covering the smaller class, as follows.

\begin{lemma}\label{lem:absorb} Let $\alpha\in (0,1)$ and let $n,m$ and $\ell\geq 1$ be integers satisfying $\ell\leq \alpha^7m/10^5$ and $\alpha^2n\geq 8m$. Let $H$ be a bipartite graph on vertex classes $A$ and $B$ such that $|A|=m$, $|B|=n$ and, for each $v\in A$, $d_H(v)\geq \alpha n$.

Then, there are disjoint subsets $B_0,B_1\subset B$ with $|B_0|=m-\ell$ and $|B_1|\geq \alpha^7n/10^5$, and the following property. Given any set $U\subset B_1$ of size $\ell$, there is a perfect matching between $A$ and $B_0\cup U$ in $H$.
\end{lemma}
\begin{proof}
Label the vertices of $A$ as $a_1,\ldots,a_m$. Note that, using Jensen's inequality,
\begin{equation}
\label{eqn:sum}
\sum_{\{i,j\}\in \binom{[m]}{2}}|N(a_i)\cap N(a_j)|=\sum_{b\in B}\binom{d_H(b)}{2}\geq n\cdot \binom{e(H)/n}{2}\geq n\cdot \binom{(m\cdot \alpha n)/n}{2}\geq \frac{\alpha^2m^2n}{4}.
\end{equation}
Now, the pairs $\{i,j\}\in \binom{[m]}{2}$ with $|N(a_i)\cap N(a_j)|< \alpha^2n/4$ collectively contribute at most $\alpha^2m^2n/8$ to the sum on the left hand side of \eqref{eqn:sum}, so that the sum over the other terms must be at least $\alpha^2m^2n/8$. Therefore, as $|N(a_i)\cap N(a_j)|\leq n$ for each $\{i,j\}\in \binom{[m]}{2}$, there are at least ${\alpha^2m^2n/8n}$ such terms. That is, letting $I$ be the set of pairs $\{i,j\}\in \binom{[m]}{2}$ with $|N(a_i)\cap N(a_j)|\geq \alpha^2n/4$, we have $|I|\geq \alpha^2m^2/8$.

For each $i=1,\ldots,m$ in turn, chose $b_i$ uniformly at random from $N(a_i)\setminus \{b_1,\ldots,b_{i-1}\}$, noting that this is possible as $d_H(a_i)\geq \alpha n>m$. Note further that, for each $\{i,j\}\in I$ with $i<j$, given any choices for $b_1,\ldots,b_{i-1}$, we have, as $m\leq \alpha^2n/8$, that
\[
\P(b_i\in N(a_j))\geq \frac{|N(a_i)\cap N(a_j)|-m}{n}\geq \frac{\alpha^2}{8},
\]
and, similarly, given any choices for $b_1,\ldots,b_{j-1}$, we have $\P(b_j\in N(a_i))\geq \alpha^2/8$. Therefore, for each $\{i,j\}\in I$, we have $\P(b_i\in N(a_j)\text{ and }b_j\in N(a_i))\geq \alpha^4/64$.

Let $K$ be the auxiliary graph with vertex set $[m]$, where $ij$ is an edge  exactly if $b_i\in N(a_j)$ and $b_j\in N(a_i)$.
Now,
\[
\E(e(K))\geq \sum_{\{i,j\}\in I}\P(b_i\in N(a_j)\text{ and }b_j\in N(a_i))\geq |I|\cdot  \alpha^4/64\geq \alpha^6m^2/512.
\]

Therefore, there is some choice of $b_1,\ldots,b_m$ for which $e(K)\geq \alpha^6m^2/512$, so that $K$ has average degree at least $\alpha^6m/256$. Now, by Theorem~\ref{thm:mader}, we can take a subgraph $K'\subset K$ which is $(\alpha^6m/10^4)$-connected. Let $U_1=V(K')$ and note that $|U_1|\geq \alpha^6m/10^4\geq 10\ell/\alpha\geq 10\ell$.
Pick a subset $U_0\subset U_1$ of size $\ell$. Let $B_0$ be the set of $b_i$'s for $i\not\in U_0$ and let $A_0$ be the set of $a_j$'s for $j\in U_1\setminus U_0$. Let $B_1$ be the set of vertices $v\in B\setminus \{b_1,\ldots,b_m\}$ such that $d_H(v, A_0)\geq \ell$. We claim that $B_0$ and $B_1$ have the desired properties.

First note that $|A_0|= |U_1|-\ell\geq \alpha^6m/10^4-\ell\geq 9\alpha^6m/10^5$. Then, as $d_H(v)\geq \alpha n$ for each $v\in A_0$, we have
\[
|B_1|\geq \frac{e(A_0,B\setminus \{b_1,\ldots,b_m\})-\ell n}{m}\geq \frac{9\alpha^6m/10^5\cdot (\alpha n-m)-\ell n}{m}\geq \frac{\alpha^7n}{10^5},
\]
as required. Now, take an arbitrary subset $U\subset B_1$ with size $\ell$, and label $U$ as $\{u_1,\ldots,u_\ell\}$. As $u_1,\ldots,u_\ell\in B_1$, we can take distinct $d_1,\ldots,d_\ell\in U_1\setminus U_0$ such that $u_i\in N_H(a_{d_i})$ for each $i\in [\ell]$.
As $K'$ is $(\alpha^6m/10^4)$-connected, and hence $\ell$-connected, by Theorem~\ref{thm:maxflow} there is a set of $\ell$ vertex-disjoint paths in $K'$ between the vertex disjoint sets $\{d_1,\ldots,d_\ell\}$ and $U_0$. Say these paths are $P_1,\ldots,P_\ell$, and direct their edges from $\{d_1,\ldots,d_\ell\}$ to $U_0$.

We are now ready to define an injection $\phi:A\to B_0\cup U$ so that $\{a\phi(a):a\in A\}$ is a matching in $H$.
For each $i\in [m]$, do the following to define $\phi(a_i)$.
\begin{itemize}
\item If $i\in [m]\setminus (\cup_{j\in [\ell]}V(P_j))$, then let $\phi(a_i)=b_i$.
\item If $i\in \cup_{j'\in [\ell]}V(P_{j'})$, and there is some $j$ such that $\vec{ji}\in \cup_{j'\in [\ell]}E(P_{j'})$, then let $\phi(a_i)=b_j$. (Note that, as $ij\in E(K)$, we have $b_j\in N_H(a_i)$.)
\item If $i\in \cup_{j'\in [\ell]}V(P_{j'})$, and there is no $j$ such that $\vec{ji}\in \cup_{j'\in [\ell]}E(P_{j'})$, then, using that $i\in \{d_1,\ldots,d_\ell\}$, find $j\in [\ell]$ such that $d_j=i$ and let $\phi(a_i)=u_j$. (Note that, by the choice of the indices $d_j$, we have that $u_j\in N_H(a_i)$.)
\end{itemize}

Note that, for each $b_j, j\in [m]$ that is not an end vertex of any of the directed paths $P_i$, $i\in [\ell]$, we have $\phi(a_{k})=b_j$ for exactly one value of $k\in [m]$, whereas for the end vertices there is no such value of $k$. The end vertices of these paths are exactly the vertices in $U_0$, and therefore every vertex in $B_0$ appears as some $\phi(a_i)$ exactly once. Furthermore, under $\phi$, the starting vertices $j$ of the paths $P_i$, $i\in [\ell]$, have the corresponding vertices $a_j$ mapped injectively by $\phi$ to $U=\{u_1,\ldots,u_\ell\}$.
Therefore, $\{a\phi(a):a\in A\}$ is a matching from $A$ to $B_0\cup U$.
\end{proof}

We now give the form of our colour absorption lemma that we use for spanning trees (though we state it more generally). As noted in Section~\ref{sketch}, the colour absorption we use for factors is very similar, but altered so that it can be used for Theorem~\ref{thm:monocopies}, and we implement this in Section~\ref{sec:factors}.

\begin{lemma}\label{lem:colourabsorb:general}
Let $1/n\ll \gamma \ll \beta \ll \alpha$ and $\delta\geq 0$. Let $H$ be a graph with $e(H)= \beta n$ and such that any $n$-vertex graph with minimum degree at least $\delta n$ contains a copy of $H$.
Let $\fG$ be a graph collection on $[n]$ with $\delta(\fG)\geq (\delta+\alpha)n$ and $|\fG|=m\geq \alpha n$.

Then, there is a copy $S$ of $H$ in $\cup_{i\in [m]}G_i$ and disjoint sets $A,C\subset [m]$, with $|A|=e(H)-\gamma n$ and $|C|\geq 10\beta m$ such that the following property holds.
Given any subset $B\subset C$ with $|B|=\gamma n$, there is a rainbow colouring of $S$ in $\fG$ using colours in $A\cup B$.
\end{lemma}
\begin{proof}
Let $G$ be the graph with vertex set $[n]$, where $e\in \binom{[n]}{2}$ is an edge of $G$ exactly when $e\in E(G_i)$ for at least $\alpha m$ values of $i\in [m]$. Then, by Proposition~\ref{prop:alsodirac}, $\delta(G)\geq \delta n$, and, therefore, $G$ contains a copy of $H$, say $S$.

Let $K$ be the bipartite graph with vertex classes $E(S)$ and $[m]$, where $ei$ is an edge exactly if $e\in G_i$. Note that, for each $e\in E(S)$, as $e\in E(G)$, we have that $d_K(e)\geq \alpha m$. Then, as $\gamma\ll \beta\ll\alpha$, by Lemma~\ref{lem:absorb} with $n=m$, $m=\beta n$ and $\ell=\gamma n$, there are disjoint sets $A,C\subset [m]$ with $|A|=e(H)-\gamma n$ and $|C|\geq 10\beta m$, such that, for any set $B\subset C$ of size $\gamma n$ there is a perfect matching between $E(S)$ and $A\cup B$. Note that for such a matching $M$, the function $\phi:E(S)\to A\cup B$, defined by $e\phi(e)\in M$ for each $e\in E(S)$, gives a rainbow colouring of $S$ in $\fG$ using colours in~$A\cup B$, as required.
\end{proof}


\section{Transversal spanning trees}\label{sec:trees}
We now embed spanning trees in a rainbow fashion using \ref{step-1} to \ref{step-5} in the outline in Section~\ref{sketch}.
In Section~\ref{sec:treepart}, we give some simple results we use to partition trees into pieces in order to carry out these steps, as well as to prove our results for these steps. In Section~\ref{sec:fixedvertex}, we deduce a version of Koml\'os, S\'ark\"ozy and Szemer\'edi's tree embedding theorem in which one vertex of the tree is embedded to a pre-specified vertex. In Section~\ref{surpluscolours}, we embed a tree using a small but linear number of surplus colours, as motivated in the proof sketch in Section~\ref{sketch}.
In Section~\ref{sec:treecolourcover}, we embed a tree while using specific colours when we have many surplus vertices, meaning that the vertex set of the tree we are embedding is smaller than the vertex subset we are trying to embed the tree to. In Section~\ref{sec:treeproof} we then carry out the outline in Section~\ref{sketch} to prove Theorem~\ref{thm:rainbowkss}.


\subsection{Partitioning trees}\label{sec:treepart}

We will use the following simple proposition to divide a tree into subtrees (see, for example, \cite[Proposition~3.22]{randomspanningtree}).

\begin{proposition}\label{littletree} Let $n,m\in \N$ satisfy $1\leq m\leq n/3$. Given any $n$-vertex tree~$T$ containing a vertex $t\in V(T)$, there are two edge-disjoint trees $T_1,T_2\subset T$ such that $E(T_1)\cup E(T_2)=E(T)$, $t\in V(T_1)$ and $m\leq |T_2|\leq 3m$.
\end{proposition}

The first implication of this we will use is that a tree can be divided into many pieces of roughly equal size, as follows.

\begin{corollary}\label{treedecomp}
Let $n,m\in \N$ satisfy $m\leq n$. Given any $n$-vertex tree~$T$, there is some $\ell\in \N$ and edge-disjoint trees $T_1,\ldots,T_\ell \subset T$ such that $\cup_{i\in [\ell]}E(T_i)=E(T)$ and, for each $i\in [\ell]$, we have $m\leq |T_i|\leq 4m$.
\end{corollary}
\begin{proof}
Let $\ell$ be the largest integer for which there are edge-disjoint trees $T_1,\ldots,T_\ell \subset T$ such that $\cup_{i\in [\ell]}E(T_i)=E(T)$ and, for each $i\in [\ell]$, we have $|T_i|\geq m$. If $|T_i|>4m$, for any $i\in [\ell]$, then by Proposition~\ref{littletree} there are two edge-disjoint trees $S_1,S_2\subset T_i$ such that $E(S_1)\cup E(S_2)=E(T_i)$ and $m\leq |S_2|\leq 3m$. Note that $|S_1|=|T_i|-|S_2|+1\geq m$. Replacing $T_i$ in $T_1,\ldots,T_\ell$ with $S_1$ and $S_2$ gives a sequence of edge-disjoint trees which contradicts the maximality of $\ell$.
\end{proof}
When embedding a spanning tree $T$, we will divide it into subtrees $T_1$, $T_2$, $T_3$ and $T_4$, before embedding them in that order in the steps laid out in Section~\ref{sketch}. This division will follow from the following proposition.
\begin{proposition}\label{prop:threesmall}
Let $n,m_1,m_3,m_4\in \N$ satisfy $1\leq m_1,m_3,m_4\leq n/10$. Given any tree $T$ with $n$ vertices, we can find four edge-disjoint trees $T_1$, $T_2$, $T_3$ and $T_4$ such that $E(T)=\cup_{i\in [4]}E(T_i)$, $m_i\leq |T_i|\leq 3m_i$ for each $i\in \{1,3,4\}$, and $T_1\cup T_2$ and $T_1\cup T_2\cup T_3$ are connected.
\end{proposition}
\begin{proof}
Using Proposition~\ref{littletree}, find edge-disjoint trees $T'$ and $T_1$ such that $E(T)=E(T')\cup E(T_1)$ and $m_1\leq |T_1|\leq 3m_1$. Let $t_1$ be the vertex in both $T_1$ and $T'$. Noting that $|T'|\geq 7n/10\geq 3m_4$, using Proposition~\ref{littletree} find edge-disjoint trees $T''$ and $T_4$
 such that $E(T')=E(T'')\cup E(T_4)$, $m_4\leq |T_4|\leq 3m_4$ and $t_1\in V(T'')$. Noting that $|T''|\geq 4n/10\geq 3m_3$, using Proposition~\ref{littletree}, find edge-disjoint trees $T_2$ and $T_3$ such that $E(T')=E(T_2)\cup E(T_3)$, $m_3\leq |T_3|\leq 3m_3$ and $t_1\in V(T_2)$.

Note that, as $t_1\in V(T_1)\cap V(T_2)$, the intersection $T_1\cup T_2$ is connected, and as $T_2\cup T_3=T''$, $T_1\cup T_2\cup T_3$ is also connected. Therefore, the trees  $T_1$, $T_2$, $T_3$ and $T_4$ satisfy the required conditions.
\end{proof}


\subsection{Embedding trees with a fixed vertex}\label{sec:fixedvertex}

We need a simple strengthening of Koml\'os, S\'ark\"ozy and Szemer\'edi's tree embedding theorem in which we have the extra condition that one vertex of the tree is embedded to a pre-specified vertex in the graph, as follows.

\begin{theorem}\label{thm:treeuncoloured}
Let $1/n\ll c \ll \alpha$. Let $G$ be an $n$-vertex graph with $\delta(G)\geq (1/2+\alpha)n$. Let $T$ be an $n$-vertex tree with $\Delta(T)\leq cn/\log n$. Let $t\in V(T)$ and $v\in V(G)$. Then, $G$ contains a copy of $T$ with $t$ copied to $v$.
\end{theorem}

To record this precisely (and to prove it without regularity), we use the following two results from a recent proof of a directed version of Koml\'os, S\'ark\"ozy and Szemer\'edi's theorem by the first author and Kathapurkar~\cite{katha}. The results we cite below can be deduced directly from Theorem~2.1 and Theorem~2.2 in \cite{katha} by applying them to a directed graph formed by adding directed edges in both directions between any vertex pair connected by an edge in a graph $G$. Additionally, in the statement of Lemma~\ref{uncolouredvertexabsorption}, we record a strengthened version of Theorem~2.1 from~\cite{katha} in which the vertex $t$ is copied is specified beforehand. This strengthened statement follows from the proof of Theorem~2.1 in~\cite{katha} as the proof begins by embedding $t$ to an arbitrary vertex of $G$.

\begin{lemma}\label{uncolouredvertexabsorption}
Let $1/n\ll c \ll \eps \ll \mu  \ll \alpha$. Let $G$ be a graph with vertex set $[n]$ with $\delta(G)\geq (1/2+\alpha)n$. Let $T$ be a tree with $\mu n$ vertices and suppose $\Delta(T)\leq cn/\log n$. Let $t\in V(T)$ and $v\in [n]$.

Then, there exists some $A\subseteq [n]$ of size $(\mu-\eps)n$ with $v\in A$ such that for any set $V\subseteq [n]\backslash A$ of size $\eps n$, the induced subgraph $G[A\cup V]$ contains a copy of $T$ in which $t$ is copied to $v$.
\end{lemma}

\begin{lemma}\label{uncolouredkss}
Let $1/n\ll c\ll \eps, \alpha$. Let $G$ be an $n$-vertex graph with $\delta(G)\geq (1/2+\alpha)n$. Let $T$ be a tree with at most $(1-\eps)n$ vertices and suppose $\Delta(T)\leq cn/\log n$. Let $t\in V(T)$ and $v\in V(G)$.

Then, there exists a copy of $T$ in $G$ in which $t$ is copied to $v$.
\end{lemma}

Using these results, we can now prove Theorem~\ref{thm:treeuncoloured}.

\begin{proof}[Proof of Theorem~\ref{thm:treeuncoloured}]  Let $\eps$ and $\mu$ satisfy $c\ll \eps \ll \mu\ll \alpha$. Using Proposition~\ref{littletree}, take two edge-disjoint trees $T_1$ and $T_2$ such that $E(T)=E(T_1)\cup E(T_2)$, $t\in V(T_2)$, and $\mu n\leq |T_1|\leq 3\mu n$. Let $s$ be the vertex common to $T_1$ and $T_2$ and assume that $V(G)=[n]$. We now embed $T$ into $G$ differently according to whether $s=t$ (Case I), $st\in E(T)$ (Case II), or neither of these hold (Case III).

\medskip

\noindent\textbf{Case I: if $s=t$.} Using Lemma~\ref{uncolouredvertexabsorption}, there is some $V_1\subseteq [n]$ of size $|T_1|-2\eps n$ with $v\in V_1$ such that the following holds. For any set $V\subseteq [n]$ with $V_1\subset V$ and $|V|=|T_1|$, the induced subgraph $G[V]$ contains a copy of $T_1$ in which $t$ is copied to~$v$. Let $V_2=\{v\}\cup ([n]\setminus V_1)$, so that $|[n]\setminus V_2|\leq |V_1|\leq 3\mu n$. Therefore, as $\mu\ll\alpha$, for each $w\in [n]$, we have $d_G(w,V_2)\geq (1/2+\alpha/2)n\geq (1/2+\alpha/2)|V_2|$.
As $|V_2|\geq n-|V_1|\geq |T_2|+\eps n$, by Lemma~\ref{uncolouredkss}, there is a copy $S_2$ of $T_2$ in $G[V_2]$ in which $t$ is copied to $v$. Using the property of $V$, there is a copy, $S_1$ say, in $G-(V(S_2)\setminus \{v\})$ of $T_1$ in which $t$ is copied to $v$.
Combining $S_1$ and $S_2$ gives the desired copy of $T$.

\medskip

\noindent\textbf{Case II: if $s\neq t$ and $st\in E(T)$.} Let $T_1'=T_1$ and let $T_2',T_3'$ be the two components of $T_2-st$, labelling them so that $s\in V(T_2')$ and $t\in V(T_3')$. Let $v_s$ be a neighbour of $v$ in $G$.
Using Lemma~\ref{uncolouredvertexabsorption} applied to $G-v$, there is some $V_1\subseteq [n]\setminus \{v\}$ of size $|T_1'|-3\eps n$ with $v_s\in V_1$ such that the following holds. For any set $V\subseteq [n]\setminus\{v\}$ with $V_1\subset V$ and $|V|=|T_1'|$, the induced subgraph $G[V]$ contains a copy of $T_1'$ where $s$ is copied to $v_s$.

Now, for each $w\in [n]$, we have $d_G(w,[n]\setminus V_1)\geq (1/2+\alpha/2)n$. Furthermore, $|[n]\setminus (V_1\cup \{v\})|\geq n-|T'_1|+3\eps n\geq |T_2'|+|T_3'|+2\eps n$. Therefore, by Lemma~\ref{lem:newdivide1}, there is some partition $[n]\setminus (V_1\cup \{v\})=V_2\cup V_3$ such that $|V_2|\geq|T_2'|+\eps n$ and $|V_3|=|T'_3|+\eps n$, and $d(w,V_i)\geq (1/2+\alpha/4)|V_i|$ for each $w\in [n]$ and $i\in \{2,3\}$.

 Therefore, by Lemma~\ref{uncolouredkss}, there is a copy $S_2$ of $T'_2$ in $G[V_2\cup\{v_s\}]$ in which $s$ is copied to $v_s$. Similarly, there is  a copy $S_3$ of $T'_3$ in $G[V_3\cup \{v\}]$ in which $t$ is copied to $v$.
  Using the property of $V$, there is a copy, $S_1$ say, in $G-((V(S_2)\setminus \{v_s\})\cup V(S_3))$ of $T_1'$ in which $s$ is copied to $v_s$. Combining $S_1$, $S_2$ and $S_3$ and the edge $v_sv\in E(G)$ gives the desired copy of $T$.

\medskip

\noindent\textbf{Case III: $s\neq t$ and $st\notin E(T)$.}
Let $s_0=s$ and label vertices so that $s_0s_1s_2$ are the initial vertices of the path from $s$ to $t$ in $T_2$. Note that this is possible as such a path has length at least 2 as $s\neq t$ and $st\notin E(T)$. Let $T_1'=T_1$. Let $T'_2$, $T_3'$ and $T_4'$ be the components of $T_2-\{s_0s_1,s_1s_2\}$ containing $s_0,s_1$ and $s_2$ respectively. Note that $t\in V(T_4')$.
 Pick $w_0\in [n]\setminus\{v\}$ arbitrarily.
Applying Lemma~\ref{uncolouredvertexabsorption} to $G-v$, there is some $V_1\subseteq [n]\setminus \{v\}$ of size $|T_1'|-5\eps n$ with $w_0\in V_1$ such that the following holds. For any set $V\subseteq [n]\setminus \{v\}$ with $V_1\subset V$ and $|V|=|T_1'|$, the induced subgraph $G[V]$ contains a copy of $T_1'$ in which $s_0$ is copied to $w_0$.

Now, for each $w\in [n]$, we have $d_G(w,[n]\setminus (V_1\cup \{v\}))\geq (1/2+\alpha/2)n$. Furthermore, $|[n]\setminus (V_1\cup \{v\})|\geq |T_2'|+|T_3'|+|T_4'|+4\eps n$. Therefore, by Lemma~\ref{lem:newdivide1}, there is some partition $[n]\setminus (V_1\cup \{v\})=V_2\cup V_3\cup V_4\cup W$ such that $|V_i|\geq |T_i'|+\eps n$ for each $i\in \{2,3,4\}$ and $|W|=\eps n$, and, for each $w\in [n]$, we have $d(w,W)\geq (1/2+\alpha /4)|W|$ and $d(w,V_i)\geq (1/2+\alpha/4)|V_i|$ for each $i\in \{2,3,4\}$.

Therefore, as $|V_4\cup\{v\}|\geq |T'_4|+\eps n$,
 there is a copy $S_4$ of $T'_4$ in $G[V_4\cup\{v\}]$ in which $t$ is copied to $v$. Let $w_2$ be the copy of $s_2$ in $S_4$.
 Using the minimum degree condition into $W$, pick $w_1\in W$ so that it is a neighbour of $w_0$ and $w_2$.
Using that $|V_2\cup\{w_0\}|\geq |T'_2|+\eps n$, and Lemma~\ref{uncolouredkss}, there is a copy $S_2$ of $T'_2$ in $G[V_2\cup\{w_0\}]$ in which $s_0=s$ is copied to $w_0$.  Similarly, using Lemma~\ref{uncolouredkss}, there is a copy $S_3$ of $T'_3$ in $G[V_3\cup\{w_1\}]$ in which $s_1$ is copied to $w_1$.
  Using the property of $V$, there is a copy, $S_1$ say, in $G-((V(S_2)\cup V(S_3)\cup V(S_4))\setminus \{w_0\})$ of $T_1'$ in which $s_0$ is copied to $w_0$. Combining $S_1$, $S_2$, $S_3$ and $S_4$ along with the path $w_0w_1w_2$ gives the desired copy of $T$.
\end{proof}


\subsection{Embedding spanning trees with surplus colours}\label{surpluscolours}
For \ref{step-2} and \ref{step-4}, we now embed a spanning tree when we have surplus colours by dividing the tree into constantly many roughly-equal pieces and iteratively embedding each piece into randomly chosen subsets in a rainbow fashion. By dividing the tree into pieces we boost the proportion of surplus colours at each iteration, allowing us to embed using the following simple lemma.

\begin{lemma}\label{extracolours}
Let $1/n\ll 1/C, c \ll \alpha$ and let $\fG$ be a graph collection on $[n]$ with $|\fG|=Cn$ and $\delta(\fG)\geq (1/2+\alpha)n$. Let $T$ be an $n$-vertex tree with $\Delta(T)\leq cn/\log n$. Let $t\in V(T)$ and $v\in[n]$.

Then, there exists a rainbow copy of $T$ in $\fG$ in which $t$ is copied to $v$.
\end{lemma}
\begin{proof} Let $G$ be the graph on $[n]$ where $e\in \binom{[n]}{2}$ is an edge exactly if $|\{i\in[Cn]\colon e\in E(G_i)\}|\geq n-1$. By Proposition~\ref{prop:alsodirac}, we have that $\delta(G)\geq (1/2+\alpha/2)n$.
By Theorem~\ref{thm:treeuncoloured}, then, $G$ contains a copy of $T$ in which $t$ is copied to $v$. As the copy of $T$ has $n-1$ edges, each of which appears in at least $n-1$ graphs $G_i$, $i\in [Cn]$, we may greedily select a distinct colour of $\fG$ for each such edge, giving a rainbow copy of $T$ in $\fG$.
\end{proof}

Having proved Lemma~\ref{extracolours} for the iterative step, we can now prove our result for \ref{step-2} and \ref{step-4}.

\begin{lemma}\label{fewextracolours}
Let $1/n\ll c\ll \eps, \alpha$. Let $\fG$ be a graph collection with vertex set $[n]$, let $|\fG|\geq (1+\eps)n$ and $\delta(\fG)\geq (1/2+\alpha)n$. Let $T$ be an $n$-vertex tree with $\Delta(T)\leq cn/\log n$. Let $t\in V(T)$ and $v\in[n]$.

Then, there exists a rainbow copy of $T$ in $\fG$ in which $t$ is copied to $v$.
\end{lemma}
\begin{proof} Let $m=|\fG|\geq (1+\eps)n$.
Let $\mu$ satisfy $c\ll \mu \ll \eps,\alpha$ and let $C=10/\sqrt{\mu}$. By Corollary~\ref{treedecomp}, we can find for some $\ell\in \N$, edge-disjoint subtrees $T_1,\ldots, T_\ell$ such that $E(T)=\cup_{i\in [\ell]}E(T_i)$, and, for each $i\in [\ell]$, it holds that $\mu n\leq |T_i|\leq 4\mu n$.
Without loss of generality, we may assume $T_1\cup\ldots \cup T_i$ is a tree for each $i\in[\ell]$, and that $t\in V(T_1)$. For each $i\in[\ell]$, let $n_i=|T_i|-1$. Note that $\sum_{i\in [\ell]}n_i=n-1$. Using Lemma~\ref{lem:newdivide1}, let $[n]\setminus\{v\}=V_1\cup \ldots\cup V_\ell$ be a partition with $|V_i|=n_i$ for each $i\in [\ell]$ and $d_{G_i}(w,V_j)\geq (1/2+\alpha/2)n_j$ for each $w\in [n]$, $j\in [\ell]$ and $i\in [m]$.

Now, noting that $|\fG|\geq C\cdot e(T_1)$ and $c\ll \mu,1/C\ll \eps,\alpha$, and using Lemma~\ref{extracolours}, embed $T_1$ to $\fG[V_1]$ in a rainbow manner so that $t$ is embedded to $v$.
 Now, suppose for some $j\in[\ell-1]$ that $T_1\cup\ldots\cup T_{j}$ has been embedded in a rainbow manner into $\fG[V_1\cup\ldots\cup V_j]$. We describe how to embed $T_{j+1}$. Let $t'$ be the vertex common to $T_{j+1}$ and $T_1\cup\ldots T_j$, and let $v'\in [n]$ be the image of $t'$ under the embedding. Note that $\delta(G_i[\{v'\}\cup V_{j+1}])\geq (1/2+\alpha/4)(1+n_{j+1})$ and there are at least $\eps n - 1\geq C\cdot e(T_{j+1})$ graphs $G_i$ that have not been used in the rainbow embedding of $T_1\cup \ldots \cup T_j$. Then, by Lemma~\ref{extracolours} there is a rainbow embedding of $T_{j+1}$ into $\fG[V_{j+1}]$ in which $t'$ is embedded to $v'$ and which uses colours not used in the embedding of $T_1\cup \ldots\cup T_j$. Thus, the embeddings combine to give a rainbow embedding of  $T_1\cup\ldots\cup T_{j+1}$ into $\fG[V_1\cup\ldots\cup V_{j+1}]$. Therefore, we can reach a rainbow embedding of $T=T_1\cup \ldots\cup T_\ell$ into $\fG$, as desired, in which, furthermore, $t$ is copied to $v$.
\end{proof}


\subsection{Covering colours}\label{sec:treecolourcover}

In this section, we prove a result that allows us to perform \textbf{Step 3} for embedding spanning trees. It is straightforward to embed a tree in a rainbow fashion using exactly a fixed set of colours in a graph collection with minimum degree higher than the number of edges of a tree, as follows.

\begin{lemma}\label{lem:treecolourcover}
Let $m,n\in \mathbb{N}$ satisfy $m\leq n-1$. Let $\fG$ be a graph collection on $[n]$ with $|\fG|=m$ and $\delta(\fG)\geq m$. Let $T$ be an $m$-edge tree. Let $t\in V(T)$ and $v\in[n]$.

Then, $\fG$ contains a rainbow copy of $T$ in which $t$ is copied to $v$.
\end{lemma}
\begin{proof}[Proof of Lemma~\ref{lem:treecolourcover}] Let $t_0=t$. By iteratively removing leaves which are not $t$, label the vertices of $V(T)\setminus \{t\}$ as $t_1,\ldots,t_m$ so that $T[\{t_0,\ldots,t_i\}]$ is a tree for each $i\in [m]$. For each $i\in [m]$, let $s_i$ be the neighbour of $t_i$ in $T[\{t_0,\ldots,t_i\}]$.

 Let $\fG=(G_1,\ldots,G_m)$ be the graph collection, and set $\phi(t_0)=v$. Greedily, for each $i\in [m]$ in turn, select $\phi(t_i)\in [n]\setminus \{\phi(t_j):j<i\}$ such that $\phi(s_i)\phi(t_i)\in E(G_i)$.
Note that this is always possible as $\delta(\fG)\geq m$. Finally, note that $\phi(T)$ with each edge $\phi(s_i)\phi(t_i)$, $i\in [m]$, given colour $i$ is a rainbow copy of $T$ in $\fG$ in which $t_0=t$ is copied to $\phi(t_0)=v$.
\end{proof}

\subsection{Proof of Theorem~\ref{thm:rainbowkss}}\label{sec:treeproof}
We can now prove Theorem~\ref{thm:rainbowkss}, using the five steps outlined in Section~\ref{sketch}. Given any set of colours $C$, we say a graph is \emph{$C$-rainbow} if its edges have been given distinct colours in $C$.

\begin{proof}[Proof of Theorem~\ref{thm:rainbowkss}] 
Let $\gamma,\beta$ be such that $c\ll \gamma \ll \beta \ll \alpha$.
By Proposition~\ref{prop:threesmall},
we can find four edge-disjoint trees $T_1$, $T_2$, $T_3$ and $T_4$ such that $T_1\cup T_2$ and $T_1\cup T_2\cup T_3$ are connected, $E(T)=\cup_{i\in [4]}E(T_i)$, and the following holds. We have $\beta n\leq |T_1|,|T_4|\leq 3\beta n$ and $\gamma n\leq |T_3|\leq 3\gamma n$. Let $m=n-1$ and label the graphs in $\fG$ so that $\fG=(G_1,\ldots,G_m)$.

\medskip
\ref{step-1}. By Lemma~\ref{lem:colourabsorb:general}, there is an uncoloured copy $S_1$ of $T_1$ in $\cup_{i\in [m]}G_i$, and disjoint sets $A,C\subset [m]$ with $|A|=e(T_1)-\gamma n$ and $|C|\geq 4\beta m$ such that the following property holds.
\begin{enumerate}[label = \textbf{P}]
\item \label{keyprop} \;Given any subset $B\subset C$ with $|B|=\gamma n$, there is an $(A\cup B)$-rainbow colouring of $S_1$ in $\fG$.
\end{enumerate}
By removing elements of $C$, assume that $|C|=e(T_4)+\gamma n$.

\medskip

\ref{step-2}.
Let $V_1=V(S_1)$ and $n_1=|T_1|$. Let $n_2=|T_2|-1$ and $n_3=|T_3|+|T_4|-2$. Note that $n-n_1=n_2+n_3$ and, for each $v\in [n]$ and $i\in [m]$,
\[
d_{G_i}(v,[n]\setminus V_1)\geq (1/2+\alpha)n-|V_1|\geq (1/2+\alpha/2)n\geq (1/2+\alpha/2)(n_2+n_3).
\]
Therefore, by Lemma~\ref{lem:newdivide1}, there is a partition $[n]\setminus V_1=V_2\cup V_3$ such that $|V_2|=n_2$, $|V_3|=n_3$, and, for each $j\in \{2,3\}$, $v\in [n]$ and $i\in [m]$,  $d_{G_i}(v,V_j)\geq (1/2+\alpha /4)n_j$.

Let $t_1$ be the vertex common to both $T_1$ and $T_2$, and let $v_1$ be the copy of $t_1$ in $S_1$. Note that, for each $v\in [n]$ and $i\in [m]$,  $d_{G_i}(v,V_2\cup\{v_1\})\geq (1/2+\alpha /8)(n_2+1)$.
Furthermore, note that
\begin{equation}\label{eqn:usefull}
|[m]\setminus (A\cup C)|=e(T)-(e(T_1)-\gamma n)-(e(T_4)+\gamma n)= e(T_2)+e(T_3)\geq |T_2|+\gamma n/2.
\end{equation}
Therefore, by Lemma~\ref{fewextracolours}, there is an $([m]\setminus (A\cup C))$-rainbow copy $S_2$ of $T_2$ in $\fG$ with vertex set $V_2\cup \{v_1\}$ in which $t_1$ is copied to $v_1$.

\medskip

\ref{step-3}. Let $B$ be the set of colours in $[m]\setminus (A\cup C)$ not used for $S_2$, and note that
\[
|B|=m-|A\cup C|-e(T_2)\overset{\eqref{eqn:usefull}}{=} (e(T_2)+e(T_3))-e(T_2)=e(T_3).
\]
Let $t_2$ be the vertex common to both $T_1\cup T_2$ and $T_3$, and let $v_2$ be the copy of $t_2$ in $S_1\cup S_2$. Note that, for each $v\in [n]$ and $i\in [m]$, we have $d_{G_i}(v,V_3\cup\{v_2\})\geq (1/2+\alpha /8)(n_3+1)$. Furthermore, $n_3\geq |T_4|\geq \beta n$, $|T_3|\leq 3\gamma n$ and $\gamma\ll \beta$. Therefore, by Lemma~\ref{lem:treecolourcover}, there is a $B$-rainbow copy $S_3$ of $T_3$ in $\fG[V_3\cup\{v_2\}]$ in which $t_2$ is copied to $v_2$.

\medskip

\ref{step-4}. Let $V_4=V_3\setminus V(S_3)$, so that $|V_4|=|T_4|-1$. Let $t_3$ be the vertex common to both $T_1\cup T_2\cup T_3$ and $T_4$, and let $v_3$ be the copy of $t_3$ in $S_1\cup S_2\cup S_3$. Note that, for each $v\in [n]$ and $i\in [m]$,  as $\gamma\ll \beta$,
\[
d_{G_i}(v,V_4\cup\{v_3\})\geq (1/2+\alpha /4)n_3-|S_3|\geq (1/2+\alpha /8)n_3\geq (1/2+\alpha /8)|T_4|.
\]
Recall that $|C|=e(T_4)+\gamma n$. Therefore, by Lemma~\ref{fewextracolours}, there is a $C$-rainbow copy $S_4$ of $T_4$ in $\fG[V_4\cup \{v_3\}]$ in which $t_3$ is copied to $v_3$.

\medskip

\ref{step-5}. Let $B'$ be the set of colours in $[m]$ not used on $S_2\cup S_3\cup S_4$, so that $A\subset B'$ and $|B'|=e(T_1)$. Using \ref{keyprop}, colour $S_1$ so that is it $B'$-rainbow. Then, $S_1\cup S_2\cup S_3\cup S_4$ is a rainbow copy of $T$ in $\fG$, as required.
\end{proof}


\section{Transversal factors}\label{sec:factors}
In this section, we prove Theorem~\ref{thm:fullyrainbowfactors} and Theorem~\ref{thm:monocopies}. We do this under one statement (Theorem~\ref{thm:unitedrainbowfactors}), using the following definition. A \emph{$t$-copy} of a graph $F$ is a copy of $F$ with an edge-colouring using exactly $t$ colours.

\begin{definition}
Let $F$ be an $r$-vertex graph with at least $t$ edges, and let $\fG$ be a graph collection on vertex set $[rn]$. An \emph{$(F,t)$-factor} in $\fG$ is a vertex-disjoint union of $t$-copies of $F$ in $\fG$, such that these $t$-copies share no colours, and every vertex appears in some $t$-copy.
\end{definition}

As an $(F,e(F))$-factor in $\fG$ is a simply a rainbow $F$-factor, and an $(F,1)$-factor is a collection of monochromatic copies of $F$ using different colours, Theorems~\ref{thm:fullyrainbowfactors} and \ref{thm:monocopies} follow immediately from the following result.

\begin{theorem}\label{thm:unitedrainbowfactors}
Let $1/n\ll \alpha,1/r$, let $F$ be a graph on $r$ vertices with at least $1$ edge, and let $t\in \{1,e(F)\}$. If $\delta_F\geq 1/2$ or $F$ has a bridge, then let $\delta^T_F=\delta_F$, and otherwise let $\delta^T_F=1/2$.

Suppose $\fG$ is a graph collection on $[rn]$ with $|\fG|=tn$. If $t=1$, then suppose $\delta(\fG)\geq (\delta_F+\alpha)rn$, while if $t=e(F)$, then suppose $\delta(\fG)\geq (\delta^T_F+\alpha)rn$.
Then, $\fG$ contains an $(F,t)$-factor.
\end{theorem}
We prove Theorem~\ref{thm:unitedrainbowfactors} using the outline in Section~\ref{sketch}. In Section~\ref{sec:facspare} we embed factors when we have many spare colours, for \ref{step-2} and \ref{step-4}. In Section~\ref{specific}, we find $(F,t)$-copies while making sure we use a set of specified colours, for \ref{step-3}. In Section~\ref{tilingproofs}, we then give the proof of Theorem~\ref{thm:unitedrainbowfactors}.



\subsection{Factors with spare colours}\label{sec:facspare}
We show here that Theorem~\ref{thm:unitedrainbowfactors} holds if the graph collection $\fG$ has linearly more elements than we need, as follows.
\begin{lemma}\label{lem:to-iterate-fully-rainbow}
Let $1/n\ll \eps,\eta,1/r$. Let $F$ be an $r$-vertex graph with at least $t$ edges, and let $m=(1+\eta) nt$. Suppose $\fG$ is a graph collection with vertex set $[rn]$, with $|\fG|=m$ and $\delta(\fG)\geq (\delta_F+\eps)rn$.
Then, $\fG$ contains an $(F,t)$-factor.
\end{lemma}
\begin{proof}
Let $k\in \N$ satisfy $1/n\ll 1/k\ll \eps,\eta,1/r$. Choose $\ell\in \N$, and $k\leq n_1,\ldots,n_{\ell}\leq 2k$ such that $n=\sum_{i\in [\ell]}n_i$. By Lemma~\ref{lem:partition}, there is a partition $[rn]=V_1\cup \ldots \cup V_\ell$ with $|V_i|=rn_i$ for each $i\in [\ell]$, and for all but at most $m/k^2\leq  \eta nt/2$ values of $j\in [m]$, we have $\delta(G_j[V_i])\geq (\delta_F+\eps/2)|V_i|$.

Let $j\leq \ell$ be the largest integer for which $\fG[V_1\cup\ldots\cup V_j]$ contains an $(F,t)$-factor, $\bar{F}$ say. Assume that $j<\ell$, for otherwise we are done. Let $I_{j+1}\subset [m]$ be the set of colours $i$ not used on $\bar{F}$ for which $\delta({G_i}[V_{j+1}])\geq (\delta_F+\eps/2)|V_{j+1}|$, noting that
\[
|I_{j+1}|\geq m-t\cdot \sum_{i\in [j]}n_i-\eta nt/2\geq m-nt-\eta nt/2 \geq \eta nt/2.
\]

Now, for each $i\in I_{j+1}$, the induced subgraph $G_{i}[V_{j+1}]$ contains an $F$-factor as $1/k\ll \eps,1/r$. There are at most $(rn_{j+1})!\leq (2rk)!$ distinct $F$-factors on $rn_{j+1}$ vertices, so as $1/n\ll \eta,1/k,1/r,1/t$, there is some vertex-disjoint union of $n_{j+1}$ copies of $F$,  say $F'$,  such that $F'\subset G_i[V_{j+1}]$ for at least $tn_{j+1}$ values of $i\in I_{j+1}$. Greedily colour the edges of $F'$ using $t$ colours from $I_{j+1}$ for each copy of $F$, and using different colours for different copies of $F$. Then, $\bar{F}\cup F'$ is an  $(F,t)$-factor of $\fG[V_1\cup\ldots\cup V_{j+1}]$, which contradicts the maximality of $j$.
\end{proof}


\subsection{Using specific colours in factors}\label{specific}
\par In this section, we find $(F,t)$-factors which exhaust a given large set of colours (for \ref{step-3} in Section~\ref{sketch}). We begin by collecting some useful auxiliary results.
The following lemma is a consequence of the standard (regularity-free) proof of the Erd\H{o}s-Stone theorem (see, for example, \cite{moderngraphtheory}). Recall that $K_r(t)$ is a complete $r$-partite graph with $t$ vertices in each vertex class.
\begin{lemma}\label{es}
Let $1/n\ll 1/T\ll 1/t\ll \eps, 1/r$. Suppose $G$ is an $n$-vertex graph with $\delta(G)\geq (1-1/(r-1)+\eps)n$, and $K$ is a copy of $K_{r-1}(T)$ contained in $G$. Then, $G$ contains a copy of $K_r(t)$ whose intersection with $K$ is a copy of $K_{r-1}(t)$.
\end{lemma}

The following lemma is a standard consequence of the K\H{o}v\'ari-S\'os-Tur\'an theorem.
\begin{lemma}\label{weakkst}
Let $1/n\ll\gamma, 1/t$. Then, any $n$-vertex graph with at least $\gamma n^2$ edges contains a copy of $K_{t,t}$.
\end{lemma}

Finally, we recall the following well known fact (see \cite{kuhnosthus}).
\begin{fact}\label{fact}
For any graph $F$ with at least 1 edge, $\delta_F> 1-1/(\chi(F)-1)$.
\end{fact}

We first prove a lemma that allows us to ensure we use one given colour while finding an $(F,t)$-copy, as follows.

\begin{lemma}\label{lem:obstinate} Let $1/n\ll 1/m\ll \eps,1/r$. Let $F$ be an $r$-vertex graph with at least 1 edge and let $t\in \{1,e(F)\}$.
Let $\fG=(G_1,\ldots,G_m)$ be a graph collection on vertex set $[n]$ with $\delta(\fG)\geq (\delta_F+\eps)n$. Suppose that $F$ has a bridge, or $\delta(\fG)\geq (1/2+\eps)n$, or $t=1$. Then $\fG$ contains a $t$-copy of $F$ which uses at least one edge from $G_1$.
\end{lemma}
\begin{proof} Note that this follows directly from the definition of $\delta_F$ when $t=1$. Suppose, then, that $t=e(F)$.
 \par Let $G$ be the graph with vertex set $[n]$ and edges exactly those which appear in at least $t$ graphs in $\fG$. We will now find an edge $e\in E(G_1)$ and a copy $F'$ of $F$ in $G+e$ containing $e$. Giving $e\in E(F')$ colour 1, and, greedily, giving each other edge of $F'$ (which is in $E(G)$) a distinct colour in $[m]\setminus \{1\}$ completes the required $(F,t)$-copy.

 It is left then to find $e$ and $F'$, which we do differently depending on whether $\delta(\fG)\geq (1/2+\eps)n$ (Case 1) or whether $F$ is a bipartite graph containing a bridge (Case 2). Note that, by Fact~\ref{fact}, if $\chi(F)\geq 3$ then $\delta_F\geq 1/2$, and thus these cases cover all eventualities.

\medskip

 \par \textbf{Case 1:} $\delta(\fG)\geq (1/2+\eps)n$.
 Then, as $1/m\ll \eps,1/t$, by Proposition~\ref{prop:alsodirac}, $\delta(G)\geq \max\{(1/2+\eps/2)n,(\delta_F+\eps/2)n\}$.
  Let $s=\chi(F)-1$, so that $\delta_F\geq 1-1/s$ by Fact~\ref{fact}. Take integers $r_1,r_2,\ldots,r_s$ so that $r_s=r$ and we have $1/n\ll 1/r_1\ll 1/r_2\ll \ldots \ll 1/r_s,\eps$.

 Now, for each $v\in [n]$,
 \[
 |N_G(v)\cap N_{G_1}(v)|\geq (1/2+\eps/2)n+(1/2+\eps)n-n=3\eps n/2,
 \]
 so that $e(G\cap G_1)\geq 3\eps n^2/4$.
Therefore, by Lemma~\ref{weakkst}, $G\cap G_1$ contains a copy of $K_2(r_1)$, say $K$. Recall that we have $\delta(G)\geq (1-1/s+\eps/2)n$. Thus, applying Lemma~\ref{es} for each $i=2,\ldots,s$ in turn, $G$ contains a copy of $K_{i+1}(r_i)$ whose intersection with $K$ is a copy of $K_2(r_i)$.

As $r_s=r$, this gives us that $G$ contains a a copy of $K_{s+1}(r)$ whose intersection with $K\subset G_1\cap G$ is a copy of $K_2(r)$. Note that this contains a copy of $F$, say $F'$, with at least one edge, $e$ say, in $K\subset G_1$, so that we have $e$ and $F'$ as required.

\medskip

 \par \textbf{Case 2:} $F$ is a bipartite graph containing a bridge.
  Then, as $1/n\ll \eps,1/t$, by Proposition~\ref{prop:alsodirac}, $\delta(G)\geq (\delta_F+\eps/2)n\geq \eps n/2$.
  Let $xy$ be a bridge of $F$, and let $F_1$ and $F_2$ each be a union of components of $F-xy$, choosing them to be vertex disjoint subgraphs and such that $F-xy=F_1\cup F_2$, $x\in V(F_1)$ and $y\in V(F_2)$. As $\delta(G)\geq (\delta_F+\eps/2)n$, the graph $G$ contains some copy of $F_1$, say $S_1$. Let $s_1$ be the copy of $x$ in $S_1$. Using that $\delta(G_1)\geq \eps n$, let $A\subset N_{G_1}(s_1)\setminus V(F_1)$ be a set with size $\eps n/4$, and let $B=[n]\setminus (V(S_1)\cup A)$. Let $G'$ be the bipartite graph with vertex classes $A$ and $B$ with edges those  in $G$ between $A$ and $B$. Note that, for each $v\in A$, it holds that $d_{G'}(v)\geq d_G(v)-|A|-|S_1|\geq \eps n/5$, so that $e(G')\geq \eps^2n^2/20$. Then, by Lemma~\ref{weakkst}, $G'$ contains a copy of $K_{r,r}$, and hence some copy of $F_2$, say $S_2$, in which $y$ is copied to some $s_2\in A$. Note that $e:=s_1s_2\in E(G_1)$, so that $F'=(S_1\cup S_2)+e$ is a copy of $F$ in $G+e$ containing $e$, as required.
\end{proof}

We use Lemma~\ref{lem:obstinate} iteratively to prove the following corollary, which we apply for \ref{step-3} from Section~\ref{sketch}.

\begin{corollary}\label{cor:to-iterate-fully-rainbow}
Let $1/n\ll \gamma \ll \eps,1/r$, let $F$ be an $r$-vertex graph and let $t\in \{1,e(F)\}$. Suppose $\fG=(G_1,\ldots,G_m)$ is a graph collection with vertex set $[rn]$, $m\geq 2t\gamma n$ and $\delta(\fG)\geq (\delta_F+\eps)rn$. Let $C\subset [m]$ satisfy $|C|\leq \gamma n$. Suppose that $t=1$, or $\delta(\fG)\geq (1/2+\eps)rn$, or $F$ has a bridge.

Then, there is a set $A\subset [rn]$ with $|A|=r\gamma n$ such that $\fG[A]$ contains an $(F,t)$-factor which uses every colour in $C$.
\end{corollary}
\begin{proof} Note that we can assume that $|C|=\gamma n$. Let $C'\subset C$ be a maximal set for which there are vertex disjoint subgraphs $F_i$, $i\in C'$, such that, for each $i\in C'$, the subgraph $F_i$ is an $(F,t)$-copy of $F$ in $\fG$ using colours in $\{i\}\cup ([m]\setminus C)$ which has some edge of colour $i$ and which uses no colours which appear on any $F_{i'}$, $i'\in C'\setminus \{i\}$.

Noting that we are done by simply letting $A=\cup_{i\in C'}V(F_i)$ if $C'=C$, for contradiction assume there is some $j\in C\setminus C'$. Let $B$ be the set of colours in $[m]\setminus C$ which are not used on any of $F_i$, $i\in C'$, and note that $|B|\geq m-t|C'|\geq t\gamma n$. Let $V=[rn]\setminus (\cup_{i\in C'}V(F_i))$, and note that $\delta(\fG[V])\geq \delta(\fG)-\eps rn/2$. Therefore, by Lemma~\ref{lem:obstinate}, there is an $(F,t)$-copy of $F$ in $\delta(\fG[V])$ using colours in $\{j\}\cup B$ which has some edge of colour $j$.
Thus, the subgraphs $F_i$, $i\in C'\cup\{j\}$, contradict the choice of $C'$.
\end{proof}

\subsection{Proof of Theorem~\ref{thm:unitedrainbowfactors}}\label{tilingproofs}
We can now prove Theorem~\ref{thm:unitedrainbowfactors}, again using the five steps in Section~\ref{sketch}.

\begin{proof}[Proof of Theorem~\ref{thm:unitedrainbowfactors}]
Let $\gamma,\beta$ and $\bar{\beta}$ satisfy $1/n\ll \gamma\ll \beta\ll \bar{\beta}\ll \alpha,1/r$.

If $t=1$, then let $\delta^T_F=\delta_F$. If $t=e(F)$, then let $\delta^T_F=1/2$ if $\delta_F\leq 1/2$ and $F$ does not have a bridge, and let $\delta^T_F=\delta_F$ otherwise. Let $\fG=(G_1,\ldots,G_{tn})$ be a graph collection with vertex set $[rn]$ for which $\delta(\fG)\geq (\delta^T_F+\alpha)rn$. To prove the theorem, it is then sufficient to find an $(F,t)$-factor in $\fG$.

Let $n_1=\beta n/2$, $n_3=\gamma n$, $n_4=\beta n$ and $n_2=n-n_1-n_3-n_4$. At each step $i$ in the outline in Section~\ref{sketch} we will find $n_i$ more vertex-disjoint copies of $F$. We begin by proving the following claim for \ref{step-1}.

\begin{claim}\label{clm:onecopy}
Given any set $U\subset [rn]$ with $|U|\leq n_1r$, there is a copy $F'$ of $F$ with vertices in $[rn]\setminus U$ such that $F'\subset G_j$ for at least $\bar{\beta} tn$ values of $j\in [tn]$.
\end{claim}
\begin{proof}[Proof of Claim~\ref{clm:onecopy}] Let $k$ satisfy $\bar{\beta}\ll 1/ k\ll \alpha,1/r$ and note that $\delta(\fG([n]\setminus U]))\geq (\delta^T_F+\alpha/2)rn$. Take a random subset $V\subset [rn]\setminus U$ with size $rk$ and let $J\subset [m]$ be the set of $j\in [tn]$ such that $\delta(G_j[V])\geq (\delta^T_F+\alpha/4)rk$. Note that, by Lemma~\ref{lem:hypergeom-app},  $\E|J|\geq tn/2$, and therefore we can find such a set $V$ for which $|J|\geq tn/2$. For each $j\in J$, there is some copy $F_j$ of $F$ in $G_j[V]$. As $\bar{\beta} \ll 1/k$, and there are at most $(rk)^{r}$ copies of $F$ with vertices in $V$, there is some copy $F'$ of $F$ in $\cup_{i\in [tn]}G_i[V]$ for which $F'\subset G_j$ for at least $\bar{\beta} tn$ values of $j\in J\subset [tn]$.
\end{proof}

We now carry out our proof using the five steps sketched in Section~\ref{sketch}.

\medskip

\noindent\ref{step-1}.  Using Claim~\ref{clm:onecopy}, iteratively find vertex-disjoint copies $F_1,\ldots,F_{n_1}$ of $F$ in $\cup_{i\in [tn]}G_i$ such that, for each $i\in [n_1]$, it holds that $F_i\subset G_j$ for at least $\bar{\beta} tn$ values of $j\in [tn]$. Let $K$ be the auxiliary bipartite graph with vertex classes $[n_1]\times [t]$ and $[tn]$, where there is an edge between $(i,i')$ and $j$ exactly if $F_i\subset G_j$.

By Lemma~\ref{lem:absorb}, there are disjoint subsets $A,C\subset [tn]$ with $|A|=tn_1-\gamma n$ and $|C|\geq 2\beta tn$, such that the following holds.
\begin{enumerate}[label = \textbf{Q}]
\item \label{keyprop2} Given any set $U\subset A\cup C$ with $A\subset U$ and $|U|=tn_1$ there is a perfect matching between $[n_1]\times [t]$ and $A\cup U$ in $K$.
\end{enumerate}
By removing elements of $C$, assume that $|C|=tn_4+\gamma n+(t-1)n_3$.

\medskip

\noindent\ref{step-2}. Let $V_1=\cup_{i\in [n_1]}V(F_i)$.
Note that, for each $v\in [n]$ and $i\in [tn]$,
\[
d_{G_i}(v,[rn]\setminus V_1)\geq (\delta^T_F+\alpha)rn-|V_1|\geq (\delta^T_F+\alpha/2)rn\geq (\delta^T_F+\alpha/2)r(n_2+n_3+n_4).
\]
Therefore, by Lemma~\ref{lem:newdivide1}, there is a partition $[n]\setminus V_1=V_2\cup V_3$ with $|V_2|=rn_2$ and $|V_3|=r(n_3+n_4)$ such that, for each $j\in \{2,3\}$, $v\in [n]$ and $i\in [tn]$, we have $d_{G_i}(v,V_j)\geq (\delta^T_F+\alpha /4)|V_j|$.
Note that
\begin{equation}\label{eqn:usefulll}
|[tn]\setminus (A\cup C)|=tn-(tn_1-\gamma n)-(tn_4+\gamma n+(t-1)n_3)=tn_2+n_3\geq (1+\gamma^2)tn_2.
\end{equation}
Therefore, by Lemma~\ref{lem:to-iterate-fully-rainbow}, $\fG[V_2]$ contains a rainbow $(F,t)$-factor $\hat{F}_2$ using colours in $[m]\setminus (A\cup C)$.

\medskip

\noindent\ref{step-3}. Let $B$ be the set of colours in $[m]\setminus (A\cup C)$ not used for  $\hat{F}_2$, and note that
\[
|B|=|[tn]\setminus(A\cup C)|-tn_2\overset{\eqref{eqn:usefulll}}{=}(tn_2+n_3)-tn_2=n_3.
\]
As, for each $v\in [n]$ and $i\in [m]$, we have $d_{G_i}(v,V_3)\geq (1/2+\alpha /4)|V_3|$, and $|C|\geq tn_4$, by Corollary~\ref{cor:to-iterate-fully-rainbow}, there is a set $V_3'\subset V_3$ with size $rn_3$ such that $\fG[V_3']$ contains a rainbow $(F,t)$-factor $\hat{F}_3$
 using colours in $B\cup C$ such that every colour in $B$ appears on at least one edge.

\medskip

\noindent\ref{step-4}. Let $V_4=V_3\setminus V(\hat{F}_3)$, so that $|V_4|=rn_4$. Note that, for each $v\in [n]$ and $i\in [m]$, as $\gamma\ll \beta$,
\[
d_{G_i}(v,V_4)\geq (1/2+\alpha /4)r(n_3+n_4)-rn_3\geq (1/2+\alpha /8)rn_4.
\]
Let $\tilde{C}$ be the set of colours in $C$ not used in $\hat{F}_3$. Note that $A\cup \tilde{C}=tn-tn_2-tn_3=tn_1+tn_4$, and therefore $|\tilde{C}|=tn_4+\gamma n$. Therefore, by Lemma~\ref{lem:to-iterate-fully-rainbow}, as $\gamma\ll \beta$, there is a rainbow $(F,t)$-factor $\hat{F}_4$ using colours in $\tilde{C}$.

\medskip

\noindent\ref{step-5}. Let $B'$ be the set of colours in $[m]$ not used on $\hat{F}_2\cup \hat{F}_3\cup \hat{F}_4$, so that $A\subset B'$ and $|B'|=tn_1$. Using \ref{keyprop2}, colour $\hat{F}_1=F_1\cup \ldots\cup F_{n_1}$
so that this forms an $(F,t)$-factor of $\fG[V_1]$ using colours in $B'$. Then, $\hat{F}_1\cup \hat{F}_2\cup \hat{F}_3\cup \hat{F}_4$ is an $(F,t)$-factor in $\fG$, as required.
\end{proof}


\section{Factors with a specified colouring}\label{sec:withspecifiedcolouring}
As mentioned in Section~\ref{sec:intro}, we now discuss the following strengthened problem for factors. Given a graph $H$ formed of $n$ vertex-disjoint copies of a fixed $r$-vertex $t$-edge graph $F$, and an injective colouring function $\phi(E(H))\to [tn]$, what minimum degree conditions on a graph collection $\fG=(G_1,\ldots,G_{tn})$ with vertex set $[rn]$ suffice to find a copy of $H$ in $\fG$ in which the copy of each edge $e\in E(H)$ has colour $\phi(e)$? Note that we consider only injective functions $\phi$ for convenience, and by including graphs multiple times in $\fG$ we can extend to any function $\phi(E(H))\to [tn]$.

To discuss this problem, let us first define the following minimum degree threshold for $F$-factors in $r$-partite graphs.

\begin{definition} Given a graph $F$ with vertex set $[r]$, let $\delta^p_F$ be the smallest real number $\delta$ such that, for each $\eps>0$, there exists $n_0$ such that, for every $n\geq n_0$, the following holds.

Let $H$ be a balanced $r$-partite graph with vertex classes $V_1,...,V_r$ of size $n$. Suppose that, for each $ij\in E(F)$, we have $\delta(H[V_i \cup V_j])\ge (\delta+\eps)n$. Then, $H$ contains an $F$-factor in which every copy of $F$ has its $i$-th vertex in $V_i$.
\end{definition}

The following result follows by modifying the techniques in Section~\ref{sec:factors} only slightly (as described below).

\begin{theorem}\label{withspecifiedcolouring}
Let $0<1/n\ll \eps,1/r$ and let $F$ be an $r$-vertex graph with $t$ edges. Let $H$ be the disjoint union of $n$ copies of $H$ and let $\phi:E(H)\to [tn]$ be a bijection.

Suppose $\fG$ is a graph collection with vertex set $[rn]$ with $|\fG|=tn$ and $\delta(\fG)\geq \left(\delta^p_F+\eps\right)rn$. Then $\fG$ contains a copy $H'$ of $H$ such that $e\in G_{\phi(e)}$ for each $e\in E(H')$.
\end{theorem}

We remark that a simple random partitioning argument shows that $\delta^p_F\geq \delta_F$, and by a result in \cite{keevash-mycroft} (see also \cite{allan}) it follows that $\delta^p_{K_r}=\delta_{K_r}=1-1/r$. A more recent result concerning cycles (see \cite{beka}) shows that $\delta^p_{C_k}=\left(1+1/k\right)/2$. More generally, the exact value of $\delta^p_F$ is largely unknown. Furthermore, we do not know in general whether the constant $\delta^p_F$ in Theorem~\ref{withspecifiedcolouring} is tight.

We now sketch the changes necessary to the proofs in Section~\ref{sec:factors} in order to prove Theorem~\ref{withspecifiedcolouring}.
For this, it is most convenient to consider how we found $F$-factors where each copy of $F$ in $\fG$ had a different colour from $[n]$ (i.e., the case of Theorem~\ref{thm:unitedrainbowfactors} when $t=1$). Instead of $n$ different colours, suppose we have $n$ \emph{patterns} $\phi_i:E(F)\to [tn]$, $i\in [n]$, where each $\phi_i$ is an injection and $\phi_i(E(F))$, $i\in [n]$, partition $[tn]$. Then, given a graph collection with vertex set $[rn]$ with $|\fG|=tn$ and $\delta(\fG)\geq \left(\delta^p_F+\eps\right)rn$, we wish to find $n$ vertex-disjoint copies, $F_1,\ldots,F_n$, of $F$ in $\fG$, so that each $e\in E(F_i)$ has colour $\phi_i(e)$.
This allows us to follow the proof of Theorem~\ref{thm:unitedrainbowfactors} with $t=1$ while replacing the $n$ colours $1,\ldots,n$ with the $n$ patterns $\phi_1,\ldots,\phi_n$. Instead of repeating this at length, we highlight only the two main adaptations required. The first is that in place of Lemma~\ref{lem:obstinate}, we need the following (much simpler) lemma to use the unused patterns outside the absorber at \ref{step-3}.

\begin{lemma}\label{lem:obstinate2} Let $1/n\ll \eps,1/r$. Let $F$ be an $r$-vertex graph with $t$ edges and let $\phi:E(F)\to [t]$ be an injection.
Let $\fG=(G_1,\ldots,G_{t})$ be a graph collection on vertex set $[rn]$ with $\delta(\fG)\geq (\delta^p_F+\eps)rn$. Then $\fG$ contains $n$ vertex disjoint copies of $F$ in which the copy of each edge $e\in E(F)$ has colour $\phi(e)$.
\end{lemma}
\begin{proof}
By Lemma~\ref{lem:newdivide1}, there is a partition $[rn]=\cup_{t\in V(F)}V_t$ such that $|V_t|=n$ for each  $t\in V(F)$ and $d_{G_j}(v,V_t)\geq (\delta^p_F+\eps/2)n$ for each  $t\in V(F)$, $v\in [n]$ and $j\in [t]$. Let $G$ be the graph with vertex set $[rn]$ where there is an edge between $v\in V_s$ and $w\in V_t$ if and only if $st \in E(F)$ and $vw\in E(G_{\phi(st)})$. Then, as $1/n\ll \eps,1/r$, by the definition of $\delta^p_F$, the graph $G$ contains an $F$-factor in which every copy of $F$ has each vertex $t\in V(F)$ copied into $V_t$. Noting that each $e\in E(F)$ lies in $G_{\phi(e)}$ shows that these copies of $F$ have the required property.
\end{proof}

The second adaptation is required in considering the colours $i$ which maintain good minimum degree in $G_i[V]$ for a random vertex set $V$. Instead, each time we need to consider which patterns $\phi_i$ have good minimum degree in $G_j[V]$ for each  colour $j$ in the pattern $\phi_i$ (i.e., for each $j\in \phi_i(E(F))$). However, as each pattern only has $t$ colours, it will still hold that a very large proportion of patterns will satisfy this. This would allow us to apply Lemma~\ref{lem:obstinate2} in place of invoking the definition of $\delta_F$ in the proof of Lemma~\ref{lem:to-iterate-fully-rainbow} and Claim~\ref{clm:onecopy}. Hence, Lemma~\ref{lem:obstinate2} also allows us to perform \ref{step-2} in the proof.

Finally, using an example related to that given in~\cite{lamaison2020colored}, we show that the minimum degree bound in Theorem~\ref{withspecifiedcolouring} needs to be larger than that used in Theorems~\ref{thm:fullyrainbowfactors} and~\ref{thm:monocopies} in certain cases. For each $\eps>0$, we show that there exist graph collections of minimum degree at least $(1-\eps)n$ which do not contain every colouring of a sufficiently large complete bipartite graph.

\begin{proposition}\label{prop:noeasygeneralisation} Let $1/n\ll 1/t \ll \eps$. There exists a graph collection $\fG$ on vertex set $[n]$ with $|\fG|=t^2$ and a function $\phi:e(K_{t,t})\to [t^2]$ such that $\delta(\fG)\ge (1-\eps)n$, and $\fG$ contains no copy of $K_{t,t}$ in which the copy of each edge $e\in E(K_{t,t})$ has colour $\phi(e)$.
\end{proposition}
\begin{proof} Let $k$ be an odd integer such that $1/t\ll 1/k\ll \eps$ and suppose $K_{t,t}$ has vertex classes $A$ and $B$ and edges $\{e_1,\ldots,e_{t^2}\}$. For each edge $e\in K_{t,t}$, select $\psi(e)$ uniformly at random from $[k]$. Given any two sets $A'\subset A$ and $B'\subset B$ such that $|A'|,|B'|\geq \ell:=t/10k$, the probability there is some $j\in [k]$ and no edge $e$ between $A'$ and $B'$ with $\psi(e)=k$ is at most $k(1-1/k)^{\ell^2}\leq k\exp(-\ell^2/k)=\exp(-\Omega(t^2/k^4))=o(e^{-2t})$ as $1/t\ll 1/k$.  As there are certainly at most $2^{2t}$ pairs of sets $A'\subset A$ and $B'\subset B$ of this size, by a union bound we can assume that $\psi:E(K_{t,t})\to [k]$ has the following property.

\begin{itemize}
    \item[\textbf{R}] For any $A'\subseteq A$ and $B'\subseteq B$ such that $|A'|,|B'|\geq t/10k$, we have $\psi(A'\times B')=[k]$ .
\end{itemize}

 We now construct a graph collection $\fG=(G_1,...,G_{t^2})$ as follows. Let $A_1,\ldots, A_{k+1}$ be a nearly balanced partition of $[n]$. Note that, as $k+1$ is even, $K_{k+1}$ admits a decomposition into $k$ perfect matchings $M_1,\ldots,M_k$. Now, for each $i\in [t^2]$, let $G_i$ be the graph obtained by removing all edges $e\in A_{j}\times A_{j'}$ where $\{j,j'\}\in M_{\psi(e_i)}$ from the complete $(k+1)$-partite graph with vertex partition $A_1,\ldots, A_{k+1}$. As $1/k\ll \eps$, it holds that $\delta(G_i)\geq (1-\eps)n$ for each $i\in [t^2]$.

For each $i\in [t^2]$, let $\phi(e_i)=i$.
Suppose, for contradiction, that  $F$ is a copy of $K_{t,t}$ in $\fG$ with vertex classes $U$ and $V$ with edges $\{f_1,\ldots,f_{t^2}\}$ such that, for each $i\in [t^2]$, the edge $f_i$ is the copy of $e_i$ and $f_i\in E(G_{\phi(i)})$. Then, by averaging, there are some $j,j'\in [k+1]$ with $|U\cap A_j|\geq t/(k+1)$ and $|V\cap A_{j'}|\geq t/(k+1)$. For some $i\in [k+1]$, we have $jj'\in M_i$.
Now, let $U',V'\subset V(K_{t,t})$ be the sets copied to $U\cap A_j$ and $V\cap A_{j'}$ respectively.
Then, by \textbf{R}, there is some $i'\in [t^2]$ for which the edge $e_{i'}$ lies between $U'$ and $V'$ and has $\psi(e_{i'})=i$. Therefore, as it is a copy of $e_{i'}$, the edge $f_{i'}$ lies between $U\cap A_j$ and $V\cap A_{j'}$ and appears in $G_{\phi(i')}=G_{i'}$. However, as $\psi(e_{i'})=i$, it must hold that $f_{i'}\in A_{j}\times A_{j'}$ and $jj'\in M_i$, a contradiction.
\end{proof}

\section*{Acknowledgements}
We thank the anonymous referees for detailed feedback that improved the presentation of the paper. The second author thanks Tibor Szabó for suggesting to investigate a version of Theorem~\ref{thm:fullyrainbowfactors}.


\begin{aicauthors}
\begin{authorinfo}[richard]
  Richard Montgomery\\
  University of Warwick\\
  Coventry, CV4 7AL, UK\\
  {\tt richard.montgomery@warwick.ac.uk}
\end{authorinfo}
\begin{authorinfo}[alp]
  Alp M\"uyesser\\
  University College London\\
  London, WC1E 6BT, UK\\
  \texttt{alp.muyesser.21@ucl.ac.uk}
\end{authorinfo}
\begin{authorinfo}[yani]
  Yani Pehova\\
  University of Warwick\\
  Coventry, CV4 7AL, UK\\
  {\tt yani.pehova@gmail.com}
\end{authorinfo}
\end{aicauthors}

\end{document}